\documentclass[11pt]{article}

\usepackage{amsmath,amsthm,amssymb}
\usepackage[usenames,dvipsnames]{xcolor}
\usepackage{enumerate}
\usepackage{graphicx}
\usepackage{cite}
\usepackage{comment}
\usepackage{mathrsfs}
\usepackage[margin=1in]{geometry}

\theoremstyle{plain}
\newtheorem{thm}{Theorem}[section]
\newtheorem{theorem}[thm]{Theorem}

\newtheorem{lemma}[thm]{Lemma}
\newtheorem{prop}[thm]{Proposition}

\newtheorem{claim}[thm]{Claim}

\def\@rst #1 #2other{#1}
\newcommand\MR[1]{\relax\ifhmode\unskip\spacefactor3000 \space\fi
  \MRhref{\expandafter\@rst #1 other}{#1}}
\newcommand{\MRhref}[2]{\href{http://www.ams.org/mathscinet-getitem?mr=#1}{MR#2}}

\theoremstyle{definition}

\numberwithin{equation}{section} 

\newcommand{\dsb}{\begin{adjustwidth}{2.5em}{0pt}
\begin{footnotesize}}
\newcommand{\dse}{\end{footnotesize}
\end{adjustwidth}}

\newcommand{\ssb}{\begin{adjustwidth}{2.5em}{0pt}}
\newcommand{\sse}{\end{adjustwidth}}

\newcommand{\aryb}{\begin{eqnarray*}}
\newcommand{\arye}{\end{eqnarray*}}
\def\alb#1\ale{\begin{align*}#1\end{align*}}
\def\allb#1\alle{\begin{align}#1\end{align}}
\newcommand{\eqb}{\begin{equation}}
\newcommand{\eqe}{\end{equation}}
\newcommand{\eqbn}{\begin{equation*}}
\newcommand{\eqen}{\end{equation*}}

\newcommand\p{\partial}
\newcommand\e{\varepsilon}

\newcommand\R{\mathbb{R}}
\newcommand\Z{\mathbb{Z}}
\newcommand\N{\mathbb{N}}

\newcommand\norm[1]{\left\Vert#1\right\Vert}

\newcommand\vphi{\varphi}

\let\originalleft\left
\let\originalright\right
\renewcommand{\left}{\mathopen{}\mathclose\bgroup\originalleft}
\renewcommand{\right}{\aftergroup\egroup\originalright}

\title{Channels of Energy for the Linearized Energy Critical Wave Equation in Even Dimensions $N\geq 8$}
 \date{ }
 \author{Andres A. Contreras Hip}
\date{}

\begin{document}

\maketitle
\begin{abstract}
We prove an exterior energy estimate for the linearized energy critical wave equation around a multisoliton for even dimensions $N\geq 8.$ This extends the result in \cite{cdkm} to higher dimensions. During the proof we encounter various additional important technical difficulties compared to lower dimensions. In particular, we need to deal with a number of generalized eigenfunctions of the static operator which increases linearly in $N.$ This makes the analysis of projections onto these eigenfunctions a higher dimensional problem, which requires linear systems to control. This is a crucial ingredient in our upcoming work where we give an alternative proof of the soliton resolution for the wave maps equation based on the method of channels of energy developed by Duyckaerts-Kenig-Merle.
\end{abstract}
\tableofcontents

\section{Introduction}

We consider radial solutions to the linearized energy critical wave equation around the ground state in even dimensions $N,$
\begin{equation}\label{LW}
\left\{
\begin{matrix}
\p_{tt} u_L - \Delta u_L + Vu_L = 0\\
u_L\vert_{t=0} = (u_0,u_1)
\end{matrix}
\right.
\end{equation}
with initial conditions $(u_0,u_1) \in \mathcal{H}=\mathcal{H}^1\times L^2(\R^N),$ where $V$ is the potential given by
\[
V=-\frac{N+2}{N-2} W^{\frac{4}{N-2}},
\]
with
\[
W= \left(1+\frac{\vert x\vert^2}{N(N-2)}\right)^{-\frac{N-2}{2}}.
\]
This is the linearized equation for
\[
\p_t^2 u - \Delta u = \vert u\vert^{\frac{4}{N-2}} u
\]
around the stationary solutions $\pm W.$ By standard arguments, \eqref{LW} is well posed in $\dot{H}^1 \times L^2$ and moreover the outer radiated energy
\[
E_{\mathrm{out}} := E_{\mathrm{out}}^- + E_{\mathrm{out}}^+, \; \lim_{t \to \pm\infty} \int_{\vert x\vert > \vert t\vert} \vert \nabla_{t,x} u_L(t,x)\vert^2 dx
\]
is well defined for $t \to \pm \infty.$

For the free wave equation,
\[
\left\{
\begin{matrix}
\p_t^2 u_F - \Delta u_F =0\\
(u_F,\p_t u_F)\vert_{t=0} = (u_0,u_1)
\end{matrix}
\right.
\]
one has the channels of energy estimate
\begin{equation}\label{basicenergychan}
\norm{(u_0,u_1)}_{\dot{H}^1\times L^2} \lesssim \sqrt{E_{\mathrm{out}}}
\end{equation}
if $N$ is odd. However, in the case that $N$ is even the full estimate fails (\cite{energypartition}), instead we have
\[
\norm{u_0}_{\dot{H}^1} \lesssim \sqrt{E_{\mathrm{out}}}
\]
if $N \equiv 0 \mod 4,$ and
\[
\norm{u_1}_{L^2} \lesssim \sqrt{E_{\mathrm{out}}}
\]
if $N \equiv 2 \mod 4.$ Extensions to domains of the form $\{\vert x\vert \geq R+\vert t\vert\}$ (obtained in \cite{corotational, ogchannels, radfields}) will also be used in the present work. Let $\Lambda = \frac{N-2}{2}+x\cdot \nabla$ denote the scaling generator. For the equation \eqref{LW}, there are two natural counterexamples to the analogue of \eqref{basicenergychan}, namely $\Lambda W,$ $t \Lambda W.$ For odd dimensions $N,$ the analogue of \eqref{basicenergychan} is
\[
\norm{\Pi_{\dot{H}^1}^\perp u_0}_{\dot{H}^1}+\norm{\Pi_{L^2}^\perp u_0}_{L^2} \lesssim \sqrt{E_{\mathrm{out}}} \mbox{ if }N \geq 3 \mbox{ is odd}
\]
where the projections $\Pi_{\dot{H}^1}^\perp,\Pi_{L^2}^\perp$ are defined by
\begin{equation}\label{Pi}
\Pi_{\dot{H}^1}^\perp = \Pi_{\dot{H}^1}(\mathrm{Span}(\Lambda W))^\perp, \; \Pi_{L^2}^\perp = \Pi_{L^2}(\mathrm{Span}(\Lambda W))^\perp.
\end{equation}
Here, if $H$ is any Hilbert space and $E$ any closed subspace, we let $\Pi_H(E)$ denote the orthogonal projection onto $E$ in $H$ and we let $\Pi_H(E)^\perp = 1- \Pi_H(E).$

In this paper, we will restrict to even dimensions $N.$ We will need to define logarithmic weakenings of $L^2.$ Let $Z_{\alpha}$ be the space defined by the norm
\[
\norm{f}_{Z_\alpha} := \sup_{R > 0} \frac{R^{-\frac{N}{2}-\alpha}}{\langle \log R\rangle} \left(\int_{R < \vert x\vert < 2R} f^2(x) dx\right)^{\frac{1}{2}}.
\]
We will also need the norm
\[
\norm{f}_{Z_{\alpha,\lambda}} = \sup_{R>0} \frac{R^{-3-\alpha}}{\inf_{1\leq j \leq J} \left\langle \frac{R}{\lambda_j}\right\rangle} \left(\int_{R \leq \vert x\vert \leq 2R} f^2 dx\right)^{\frac{1}{2}},
\]
where $\lambda = (\lambda_1, \ldots \lambda_J)$ is such that $\lambda_1 \geq \cdots \geq \lambda_J.$

We have the following result.
\begin{theorem}\label{single}
There exists a constant $C$ such that for any radial solution $u_L$ of \eqref{single} we have
\[
\norm{\Pi_{\dot{H}^1}^\perp u_0}_{\dot{H}^1} + \norm{\Pi_{L^2}^\perp u_1}_{Z_{-3}} \leq C\sqrt{E_{\mathrm{out}}},
\]
if $N \equiv 0 \mod 4,$
\[
\norm{\Pi_{L^2}^\perp u_1}_{L^2} + \norm{\nabla \Pi_{\dot{H}^1}^\perp u_0}_{Z_{-3}} \leq C\sqrt{E_{\mathrm{out}}},
\]
if $N \equiv 2 \mod 4.$
\end{theorem}
We will also consider the linearized energy critical wave equation around a multisoliton. To define this, let $J \in \N$ and let
\[
\Lambda_J := \{\lambda = (\lambda_1, \ldots , \lambda_J)\in (0,\infty)^J, \lambda_J < \lambda_{J-1} < \cdots < \lambda_1\}.
\]
Let
\[
V_\lambda = \sum_{j = 1}^J V_{\{\lambda_j\}},
\]
where $f_{\{\lambda\}}$ denotes the scaling
\[
f_{\{\lambda\}}(x) := \frac{1}{\lambda^2}f\left(\frac{x}{\lambda}\right).
\]
We define the scalings
$f_{(\lambda)}$ as the $\dot{H}^1$ critical scaling
\[
f_{(\lambda)}(x) = \frac{1}{\lambda^{\frac{N-2}{2}}} f\left(\frac{x}{\lambda}\right).
\]
We define the scalings
$f_{(\lambda)}$ as the $L^2$ critical scaling
\[
f_{[\lambda]}(x) = \frac{1}{\lambda^{\frac{N}{2}}} f\left(\frac{x}{\lambda}\right).
\]
We will also consider the equation
\begin{equation}\label{multieq}
\left\{
\begin{matrix}
\p_{tt}u_L - \Delta u_L + V_{\lambda} u_L = 0,\\
u_L\vert_{t=0} = (u_0,u_1),
\end{matrix}
\right.
\end{equation}
which is the energy critical wave equation linearized around a multisoliton.

Let $\gamma(\lambda)$ be defined by
\[
\gamma(\lambda) = \sup_{1\leq j \leq J-1}\frac{\lambda_j^2}{\lambda_{j+1}^2}.
\]
We have the following statement for the multisoliton case.

\begin{theorem}\label{mainmulti}
There exists a dimensional constant $C$ such that for any radial solution $u_L$ of \eqref{multieq} we have
\[
\norm{\Pi_{\dot{H}^1}^\perp u_0}_{L^2} + \norm{\Pi_{L^2}^\perp u_1}_{Z_{-3,\lambda}} \leq C\left(\sqrt{E_{\mathrm{out}}}+\gamma(\lambda)\norm{(u_0,u_1)}_{\dot{H}^1\times L^2}\right),
\]
if $N \equiv 0 \mod 4,$
\[
\norm{\Pi_{L^2}^\perp u_1}_{L^2} + \norm{\nabla \Pi_{\dot{H}^1}^\perp u_0}_{Z_{-3,\lambda}} \leq C\left(\sqrt{E_{\mathrm{out}}}+\gamma(\lambda)\norm{(u_0,u_1)}_{\dot{H}^1\times L^2}\right),
\]
if $N \equiv 2 \mod 4.$
\end{theorem}
Theorems \ref{single} and \ref{mainmulti} were proven in \cite{mainref} for the case of dimensions $N=6,8.$ However, in higher dimensions, substantial additional work is needed. In the one soliton case, to bound $\norm{\Pi_{\dot{H}^1}^\perp u_0}_{\dot{H}^1}$ if $N \equiv 0 \mod 4$ (respectively $\norm{\Pi_{L^2}^\perp u_0}_{L^2}$ if $N \equiv 2 \mod 4$) one needs to use a channels of energy estimate for the free wave equation combined with a contradiction argument to show that any asymptotic counterexample to Theorem \ref{single} is trivial (see Lemma \ref{notinterestingsingle}). Additionally, in the multisoliton case (Theorem \ref{mainmulti}), one proceeds by induction on the number of solitons. An additional argument is used in the analogous contradiction argument to control the orthogonal projection onto the span of the generalized eigenfunctions of $-\Delta +V$ corresponding to $\Lambda W.$

Theorem \ref{mainmulti} will be used in \cite{new} to provide an alternative proof to that of Jendrej and Lawrie in \cite{JL} of the soliton resolution for radial equivariant wave maps on $\mathbb{S}^2$ in even dimensions.

\medskip
\noindent\textbf{Acknowledgements.} The author would like to thank Carlos Kenig for suggesting the problem, useful discussions and comments, and for carefully reading earlier versions of this manuscript.

\section{More general nonlinearities}

Our main result also holds for other similar equations. Let us consider the equation
\begin{equation}\label{othereqn}
\p_t^2 u - \Delta u + \vphi(r,u) = 0
\end{equation}
for $r\in \R^+,$ $t \in \R.$ Suppose that $\vphi$ is of the form
\begin{equation}\label{vphiform}
\vphi(r,u) = \sum_{k =2}^\infty \vphi_k r^{(k-1)\left(\frac{N}{2}-1\right)-2} u^k,
\end{equation}
where
\[
\lim_{k\to \infty} \tau^k \vphi_k = 0
\]
for any $\tau>0.$ We will make some additional assumptions:
\begin{itemize}
\item[(A1)] If $u$ is a solution to \eqref{othereqn}, then so is $u_\lambda(r,t) = \lambda^{\frac{N-2}{2}} u(\lambda r , \lambda t).$
\item[(A2)] All static radial solutions $U(r)$ of \eqref{othereqn} verify
\[
\p_r^\alpha U(r) \sim r^{-(N-2)-\alpha}
\]
as $r \to \infty$ for $\alpha \in \N.$
\end{itemize}
One can then consider the equation \eqref{othereqn} linearized around $U(r),$
\begin{equation}\label{otherlin}
\p_t^2 u-\Delta u + \p_2 \vphi (r,U)u=0,
\end{equation}
where $\p_2$ denotes the derivative with respect to the second variable of $\vphi$ and where $U$ is any static solution to \eqref{othereqn}.

We have the following result.
\begin{theorem}\label{othersinglethm}
Suppose $\vphi$ is a function satisfying assumptions (A1) and (A2) and is of the form \eqref{vphiform}. Then for any radial solution $u$ of \eqref{otherlin}, we have the channels of energy estimate
\[
\norm{\Pi_{\dot{H}^1}^\perp u_0}_{L_2} + \norm{\Pi_{L^2}^\perp u_1}_{Z_{-3}} \leq C\sqrt{E_{\mathrm{out}}},
\]
if $N \equiv 0 \mod 4,$
\[
\norm{\Pi_{L^2}^\perp u_1}_{L_2} + \norm{\nabla \Pi_{\dot{H}^1}^\perp u_0}_{Z_{-3}} \leq C\sqrt{E_{\mathrm{out}}},
\]
if $N \equiv 2 \mod 4.$
\end{theorem}
The proof of Theorem \ref{othersinglethm} is the same as that of Theorem \ref{single}. The proof of Theorem \ref{single} only uses the asymptotics of $U.$

Similarly, we have the follwoing analogue in the case of a multisoliton.
\begin{theorem}
Suppose $\vphi$ is a function satisfying assumptions (A1) and (A2) and is of the form \eqref{vphiform}. Also, suppose that $\lambda = (\lambda_1, \ldots , \lambda_J)$ is such that $\gamma(\lambda) \leq \gamma_0$ for some small enough $\gamma_0.$ Then there exists a constant $C$ such that for any radial solution $u$ of
\[
\left\{
\begin{matrix}
\p_t^2 u -\Delta u + \p_2 \left(\vphi(r,U_{\lambda_1}) + \cdots + \vphi(r,U_{\lambda_J})\right) u = 0\\
(u,\p_t u)\vert_{t=0} = (u_0,u_1),
\end{matrix}
\right.
\]
we have
\[
\norm{\Pi_{\dot{H}^1}^\perp u_0}_{L_2} + \norm{\Pi_{L^2}^\perp u_1}_{Z_{-3,\lambda}} \leq C\left(\sqrt{E_{\mathrm{out}}}+\gamma(\lambda)\norm{(u_0,u_1)}_{\dot{H}^1\times L^2}\right),
\]
if $N \equiv 0 \mod 4,$
\[
\norm{\Pi_{L^2}^\perp u_1}_{L_2} + \norm{\nabla \Pi_{\dot{H}^1}^\perp u_0}_{Z_{-3,\lambda}} \leq C\left(\sqrt{E_{\mathrm{out}}}+\gamma(\lambda)\norm{(u_0,u_1)}_{\dot{H}^1\times L^2}\right),
\]
if $N \equiv 2 \mod 4.$
\end{theorem}

To see that the hypotheses of Theorem \ref{othersinglethm} are not vacuous, consider $k$-equivariant wave maps from Minkowski space into the sphere. Suppose that $\psi(r,t)$ solves 
\[
\p_t^2 \psi - \p_r^2 \psi - \frac{1}{r}\p_r \psi + k^2\frac{\sin(2\psi)}{2r^2} = 0.
\]
This equation has the static solutions
\[
Q_{k,\ell,\lambda}(r) = \ell k + 2\arctan(\lambda r^k)
\]
for all $\ell \in \Z.$ If we set $u$ to be defined by
\[
u(r,t) = \frac{\psi(r,t)}{r^k},
\]
then $u$ solves 
\[
\p_t^2 u -\Delta u + \vphi(r,u) = 0,
\]
where
\[
\vphi = \frac{k^2}{r^{2+k}}\left(r^k u - \frac{\sin(2r^k u)}{2}\right),
\]
which satisfies the hypotheses of the theorem.

\section{Preliminaries}
Let $\norm{u}_{L_R^p\left(\R^N\right)}$ and $\norm{u}_{\dot{H}^1_R\left(\R^N\right)}$ be the norms defined by
\[
\norm{u}_{L_R^p\left(\R^N\right)} := \left(\int_R^\infty \vert u(r)\vert^p r^{N-1} dr\right)^{\frac{1}{p}},
\]
\[
\norm{u}_{\dot{H}^1_R\left(\R^N\right)} := \left(\int_R^\infty \vert \p_r u(r)\vert^2 r^{N-1} dr\right)^{\frac{1}{2}}.
\]
Similarly, we define the exterior Strichartz norms
\[
\norm{u}_{L_t^pL_r^q(r>R+\vert t-\bar{t}\vert)} := \left(\int_\R \left(\int_{R+\vert t-\bar{t}\vert} \vert u(\bar{t},r)\vert^q r^{N-1}dr\right)^{\frac{p}{q}}d\bar{t}\right)^{\frac{1}{p}}
\]
%Same extension as in Remark 2.3 still works, and is completely independent of dimension

In the case that $N\equiv 0 \mod 4,$ we define the projection $\pi_0$ as the orthogonal projection with respect to the $\dot{H}_R^1$ inner product to the subspace
\begin{equation}\label{span1}
\mathrm{Span}\left\{\frac{1}{\vert x\vert^{N-2k}}:1\leq k \leq \frac{N}{4}\right\}.
\end{equation}
Similarly, if $N\equiv 2 \mod 4,$ we define the projection $\pi_1$ as the orthogonal projection with respect to the $L_R^2$ inner product to the subspace
\begin{equation}\label{span2}
\mathrm{Span}\left\{\frac{1}{\vert x\vert^{N-2k}}:1\leq k \leq \frac{N-2}{4}\right\}.
\end{equation}
Additionally, we define the alternating norm
\[
\norm{u}_{\mathcal{A},N,R} =
\left\{
\begin{matrix}
\norm{\pi_0^\perp u}_{\dot{H}^1_R} & \mbox{ if } N \equiv 0 \mod 4\\
\norm{\pi_1^\perp u}_{L_R^2} & \mbox{ if } N \equiv 2 \mod 4
\end{matrix}
\right.
\]
where
\[
\pi_0^\perp u_0 = u_0-\pi_0 u_0
\]
and
\[
\pi_1^\perp u_1 = u_1-\pi_1 u_1.
\]
Let $\mathcal{C}_{t,R}$ be the exterior cone
\[
\mathcal{C}_{t,R} :=\{(\bar{t},\bar{r}) \in \R\times (0,\infty): \bar{r} \geq R+\vert \bar{t}-t\vert \}.
\]
We also define the exterior energy
\[
E_{t,R}=\sup_{\bar{t}\in \R} \left(\norm{(u,\p_tu)}_{\dot{H}_{R+\vert t-\bar{t}\vert}^1 \times L_{R+\vert t-\bar{t}\vert}^2} \right).
\]
We have the following exterior Strichartz estimate.
\begin{lemma}\label{extstrichartz}%analogue of Lemma 2.2 in [CDKM].
Suppose that $N>3$ and that $(u_0,u_1)\in \dot{H}^1(\R^N)\times L^2(\R^N)$ is radial, and $u$ is the solution to
\[
\p_{tt}u - \Delta u = f
\]
with initial conditions $(u,\p_tu)\vert_{t=t_0}=(u_0,u_1),$ where $f \in L_t^1L_r^2$ is radial. Then for any $R_0>0,$ $t_0\in \R$ we have
\[
\norm{(u,\p_t u)}_{E_{t_0,R_0}} + \norm{u}_{L_t^2L_r^{\frac{2N}{N-3}}(C_{t_0,R_0})} \lesssim \norm{(u_0,u_1)}_{\dot{H}^1\times L^2}+\norm{f}_{L_t^1L_r^2(\mathcal{C}_{t_0,R_0})}.
\]
\end{lemma}
\begin{proof}
Without loss of generality, suppose that $t_0=0.$ Recall that
\begin{equation}\label{keeltaostrich}
\sup_{t \in \R} \norm{(u(t),\p_tu(t))}_{\dot{H}^1\times L^2} + \norm{u}_{L_t^2L_r^{\frac{2N}{N-3}}} \lesssim \norm{(u_0,u_1)}_{\dot{H}^\times L^2} +\norm{f}_{L_t^1L_r^2}
\end{equation}
(see \cite[corollary 1.3]{keeltao}). Let $(\tilde{u}_0,\tilde{u}_1)$ be any extension of $(u_0,u_1)1_{r>R_0}$ such that $(\tilde{u}_0,\tilde{u}_1)\in \dot{H}^1\times L^2$ with
\[
\norm{(\tilde{u}_0,\tilde{u}_1)}_{\dot{H}^1\times L^2} \lesssim \norm{(u_0,u_1)}_{\dot{H}_{R_0}^1\times L^2_{R_0}}.
\]
We also let $\tilde{f} = f 1_{C_{0,R_0}}.$ Note that $\tilde{f} \in L_t^1L_r^2.$ Letting $\tilde{u}$ be the solution to
\[
\left\{
\begin{matrix}
\p_t^2 \tilde{u} - \Delta \tilde{u} = \tilde{f},\\
(u,\p_u)\vert_{t=t_0} = (\tilde{u}_0,\tilde{u}_1)
\end{matrix}
\right.
\]
we see that by finite propagation speed we have $u(x,t) = \tilde{u}(x,t)$ for all $(x,t)\in C_{0,R_0}.$ Then we have by \eqref{keeltaostrich} that
\begin{eqnarray*}
\norm{(u,\p_t u)}_{E_{t_0,R_0}} + \norm{u}_{L_t^2L_r^{\frac{2N}{N-3}}}&\lesssim& \sup_{t \in \R} \norm{(\tilde{u}(t),\p_t\tilde{u}(t))}_{\dot{H}^1\times L^2} + \norm{\tilde{u}}_{L_t^2L_r^{\frac{2N}{N-3}}}\\
&\lesssim& \norm{(\tilde{u}_0,\tilde{u}_1)}_{\dot{H}^\times L^2} +\norm{\tilde{f}}_{L_t^1L_r^2}\\
&\lesssim &\norm{(u_0,u_1)}_{\dot{H}^1\times L^2}+\norm{f}_{L_t^1L_r^2(\mathcal{C}_{t_0,R_0})}.
\end{eqnarray*}
This completes the proof.
\end{proof}
We will also need the following exterior energy estimate from \cite{higherchan}.
\begin{prop}\label{2.3}
Suppose $N$ is even, and let $u\in \R\times \R^N$ solve
\[
\p_{tt}u-\Delta u = f,
\]
with initial conditions $(u,\p_t u)\vert_{t=0}=(u_0,u_1)$ where $u_0,u_1,$ and $f$ are radial. Then we have
\begin{equation}\label{extest}
\norm{u}_{\mathcal{A},N,R} \lesssim \lim_{t \to \infty} \norm{(u(t),\p_tu(t))}_{\dot{H}_{R+\vert t\vert}^1(\R^N)\times L_{R+\vert t\vert}^2(\R^N)} + \norm{f}_{L_t^1L_r^2(\mathcal{C}_{0,R})}.
\end{equation}
\end{prop}

\section{Single soliton case (non resonant component)}\label{firsthalfsec}

We recall that $\Lambda W$ solves
\[
-\Delta(\Lambda W)+ V(\Lambda W) = 0.
\]
We will say that $f(r)\sim g(r)$ if $r \to a,$ with $a \in \R\cup\{\pm \infty\}$ if $\frac{f(r)}{g(r)}\to C$ for some constant $C$ as $r \to a.$
\begin{lemma}
There exists a solution $\Gamma$ to
\[
-\Delta \Gamma + V \Gamma = 0,
\]
where the Wronskian is given by
\[
\Gamma(r)\p_r(\Lambda W)(r) - \Lambda W(r)\p_r\Gamma(r) =r^{-(N-1)}
\]
and $\Gamma$ has the asymptotics
\[
\Gamma(r)\sim c_1r^{-(N-2)}\quad \mbox{as }r\to0, \; \Gamma(r)\sim c_2\quad \mbox{as }r\to\infty
\]
for some dimensional constants $c_1,c_2.$
\end{lemma}
\begin{proof}
We let
\[
\Gamma(r) = -\Lambda W(r)\int_1^r s^{-(N-1)}(\Lambda W)^{-2}(s)ds.
\]
Then we have
\[
-\Delta \Gamma + V \Gamma = 0,
\]
and for all $r>0,$
\[
\Gamma(r)\p_r(\Lambda W)(r) - \Lambda W(r)\p_r\Gamma(r) =r^{-(N-1)}.
\]
Since
\[
\Lambda W(r) = \frac{N-2}{2}\left(\frac{1-\frac{r^2}{N(N-2)}}{1+\frac{r^2}{N(N-2)}}\right)W(r),
\]
we have
\begin{equation}\label{Lambdaasymp}
\Lambda W \sim \frac{N-2}{2} \quad \mbox{as } r\to0,\; \Lambda W \sim -\frac{N-2}{2}\frac{(N(N-2))^{\frac{N-2}{2}}}{r^{N-2}} \quad \mbox{as }r \to \infty
\end{equation}
and so
\begin{equation}\label{gammaasymp}
\Gamma(r)\sim c_1r^{-(N-2)}\quad \mbox{as }r\to0, \; \Gamma(r)\sim c_2\quad \mbox{as }r\to\infty.
\end{equation}
\end{proof}

Let $T_0^\infty:=\Lambda W,$ $T_0^0:=\Gamma.$

\begin{prop}\label{varofpars}
Suppose that $N\geq 6.$ There exist $T_k^\infty, T_k^0$ for $1\leq k \leq \frac{N-6}{2}$ such that
\[
-\Delta T_{k+1}^0 + VT_{k+1}^0 = -T_k^0
\]
\[
-\Delta T_{k+1}^\infty + VT_{k+1}^\infty = -T_k^\infty
\]
and we have the asymptotics
\begin{equation}\label{asympind1}
T_k^0(r)\sim r^{a_k^0}\quad \mbox{as }r\to0,\;T_k^0(r)\sim r^{b_k^0}\quad \mbox{as }r\to\infty,
\end{equation}
\begin{equation}\label{asympind2}
T_k^\infty(r)\sim r^{a_k^\infty}\quad \mbox{as }r\to0,\;T_k^\infty(r)\sim r^{b_k^\infty}\quad \mbox{as }r\to\infty,
\end{equation}
where $a_k^0,b_k^0,a_k^\infty,b_k^\infty$ are given by
\[
a_k^0 =-N+2+2k, b_k^0 = 2k,
\]
\[
a_k^\infty= -N+2, b_k^\infty=-N+2+2k
\]
if $k > 0,$ and if $k=0$ we have
\[
a_0^0 = -N+2,\; b_0^0 = 0,\; a_0^\infty = 0, \; b_0^\infty = -N+2.
\]
\end{prop}
\begin{proof}
We prove the result by induction. Suppose by induction that \eqref{asympind1}, \eqref{asympind2} hold for $k.$ Let $\mathcal{S}_0, \mathcal{S}_\infty$ denote the operators defined by
\[
\mathcal{S}_0f(r) :=T_0^0(r)\int_0^r \rho^{N-1} T_0^\infty(\rho) f(\rho)d\rho+T_0^\infty(r) \int_1^r \rho^{N-1} T_0^0(\rho)f(\rho)d\rho.
\]
\[
\mathcal{S}_\infty f(r) := -T_0^0(r)\int_r^\infty \rho^{N-1} T_0^\infty(\rho) f(\rho)d\rho+T_0^\infty(r) \int_1^r \rho^{N-1} T_0^0(\rho)f(\rho)d\rho.
\]
Then let
\[
T_k^\infty = -\mathcal{S}_\infty (T_{k-1}^\infty), \;T_k^0 = -\mathcal{S}_0(T_{k-1}^0).
\]
Note that if $f\sim r^\alpha$ as $r \to0,$ $f\sim r^\beta$ as $r \to\infty,$ then
\[
\mathcal{S}_0 f \sim r^{\alpha+2}, \mbox{ as }r \to 0,\; \mathcal{S}_0 f \sim C+r^{\beta+2}\mbox{ as }r \to \infty
\]
if $\alpha > -N,$ $\alpha+2 < 0,$ $\beta>-2.$ Similarly,
\[
\mathcal{S}_\infty f \sim r^{\alpha+2} + r^{-(N-2)} + C \mbox{ as }r \to 0,\; \mathcal{S}_\infty f \sim \frac{1}{r^{N-2}}+r^{2+\beta} \mbox{ as }r\to \infty
\]
if $\alpha>-N,$ $\beta < -2.$

Therefore
\[
T_{k+1}^0 \sim r^{\min\{-(N-2),a_k^0+2\}}\quad \mbox{as }r\to0,\;  T_{k+1}^0 \sim r^{b_k^0+2}\quad \mbox{as }r\to\infty
\]
and
\[
T_{k+1}^\infty \sim r^{\min\{-(N-2),a_k^\infty+2\}}\quad \mbox{as }r\to0,\;  T_{k+1}^\infty \sim r^{b_k^\infty+2}\quad \mbox{as }r\to\infty.
\]
By using the induction hypothesis this completes the proof, since \eqref{Lambdaasymp}, \eqref{gammaasymp} prove the case $k=0.$
\end{proof}
\begin{lemma}\label{decompo}
Let $T_k^\infty, T_k^0$ be defined as in Lemma \ref{varofpars}. There exist coefficients $e_k^i$ such that
\[
\tilde{T}_k^0 := T_k^0 - \left(e_0^k T_0^\infty+\cdots + e_{k+1}^kT_{k+1}^\infty\right) \sim r^2 \mbox{ as } r\to 0.
\]
\end{lemma}
\begin{proof}
We will prove this by induction. Note that
\[
0 <\int_0^\infty \rho^{N-1} (T_0^\infty(\rho))^2 d\rho < \infty,
\]
since $T_0^\infty \sim 1$ as $r \to 0$ and $T_0^\infty \sim r^{-(N-2)}$ as $r\to \infty.$ Letting
\[
e_1^0 : = -\left(\int_0^\infty \rho^{N-1} (T_0^\infty(\rho))^2 d\rho \right)^{-1}
\]
and
\[
e_0^0 = e_1^0 \int_0^1 \rho^{N-1}T_0^0(\rho)T_0^\infty(\rho)d\rho
\]
which is well defined since $\rho^{N-1}T_0^0(\rho)T_0^\infty(\rho) \sim \rho$ for $\rho \to 0,$ we see that
\begin{eqnarray*}
T_0^0 - e_1^0 T_1^\infty &=& T_0^0 - e_1^0\left((e_1^0)^{-1}T_0^0+T_0^0(r)\int_0^r \rho^{N-1} (T_0^\infty(\rho))^2 d\rho+T_0^\infty(r) \int_1^r \rho^{N-1} T_0^0(\rho)T_0^\infty(\rho)d\rho\right)\\
&=& -e_1^0\left(T_0^0(r)\int_0^r \rho^{N-1} (T_0^\infty(\rho))^2 d\rho+T_0^\infty(r) \int_1^r \rho^{N-1} T_0^0(\rho)T_0^\infty(\rho)d\rho\right)
\end{eqnarray*}
and thus
\begin{eqnarray*}
T_0^0 - e_0^0 T_0^\infty - e_1^0 T_1^\infty & = & -e_1^0\left(T_0^0(r)\int_0^r \rho^{N-1} (T_0^\infty(\rho))^2 d\rho+T_0^\infty(r) \int_0^r \rho^{N-1} T_0^0(\rho)T_0^\infty(\rho)d\rho\right)\\
&\sim & r^{-(N-2)} \int_0^r \rho^{N-1} d\rho + \int_0^r \rho d\rho \sim r^2.
\end{eqnarray*}
This proves the base case. For the general case, suppose that $e_i^k$ for $0 \leq i \leq k+1$ have all been defined. Then we have
\begin{eqnarray*}
T_{k+1}^0 &=& -\mathcal{S}_0(T_k^0)\\
&=& \mathcal{S}_0 \left(\tilde{T}_k^0 + e_0^k T_0^\infty + \cdots + e_{k+1}^k T_{k+1}^\infty\right)\\
& = & \mathcal{S}_0(\tilde{T}_k^0) + \sum_{i=0}^{k+1} e_i^k \mathcal{S}_\infty (T_i^\infty) +  \sum_{i=0}^{k+1} e_i^k (\mathcal{S}_0-\mathcal{S}_\infty)T_i^\infty\\
&=& \mathcal{S}_0(\tilde{T}_k^0) - \sum_{i=1}^{k+2} e_{i-1}^k T_i^\infty + \sum_{i=0}^{k+1} e_i^k T_0^0(r) \int_0^\infty \rho^{N-1} T_0^\infty(\rho) T_k^\infty(\rho) d\rho,
\end{eqnarray*}
where we note that $\mathcal{S}_0(\tilde{T}_k^0)$ is well defined by the induction hypothesis and where
\[
\int_0^\infty \rho^{N-1} T_0^\infty(\rho) T_k^\infty(\rho) d\rho
\]
is well defined by Lemma \ref{varofpars}. This implies that
\begin{eqnarray*}
T_{k+1}^0 &=& \mathcal{S}_0(\tilde{T}_k^0) + \left(\sum_{i=0}^{k+1} e_i^k \int_0^\infty \rho^{N-1} T_0^\infty(\rho) T_k^\infty(\rho) d\rho\right) \tilde{T}_0^0\\
&+& \left(\sum_{i=0}^{k+1} e_i^k \int_0^\infty \rho^{N-1} T_0^\infty(\rho) T_k^\infty(\rho) d\rho\right)(e_0^0T_0^\infty+e_1^0 T_1^\infty) - \sum_{i=1}^{k+2} e_{i-1}^k T_i^\infty.
\end{eqnarray*}
Noting that the first line in the equation above is of order $r^2$ as $r \sim 0,$ this completes the proof.
\end{proof}
From the solutions obtained from Proposition \ref{varofpars}, we obtain solutions of the homogeneous equation $\p_{tt}u-\Delta u +Vu=0$ such that
\[
\lim_{t \to \pm \infty} \int_{\vert x\vert > R+\vert t\vert}\vert \nabla_{t,x}u(t,x)\vert^2 dx = 0
\]
for any $R>0,$ namely
\[
S_k^{\infty,0}:= \sum_{i=0}^k T_{k-i}^\infty \frac{t^{2i}}{(2i)!},
\]
\[
S_k^{\infty,1}:=\sum_{i=0}^k T_{k-i}^\infty\frac{t^{2i+1}}{(2i+1)!}.
\]
Note that these are solutions for the systems
\[
\left\{
\begin{matrix}
\p_{tt}S_k^{\infty,0}-\Delta S_k^{\infty,0} +VS_k^{\infty,0}=0\\
(S_k^{\infty,0},\p_tS_k^{\infty,0})\vert_{t=0} = (T_k^\infty,0)
\end{matrix}
\right.
\]
\[
\left\{
\begin{matrix}
\p_{tt}S_k^{\infty,1}-\Delta S_k^{\infty,1} +VS_k^{\infty,1}=0\\
(S_k^{\infty,1},\p_tS_k^{\infty,1})\vert_{t=0} = (0,T_k^\infty)
\end{matrix}
\right.
\]
Suppose that $N \geq 8,$ and let
\[
\mathcal{D}:=\left\{S_k^{\infty,0}:0\leq k \leq \frac{N-6}{2}\right\} \cup \left\{S_k^{\infty,1}:0\leq k \leq \frac{N-8}{2}\right\}
\]
We claim these are the only ones.
\begin{prop}\label{rigidity}
Let $R>0$ and suppose that $u$ is a non-radiative solution to \eqref{LW} with $N\geq 8.$ Suppose additionally that
\begin{itemize}
    \item[i)] $u$ is an even in time solution in the case that $N\equiv 0 \mod 4;$
    \item[ii)]$u$ is an odd in time solution in the case that $N\equiv 2 \mod 4.$
\end{itemize}
Suppose that $R>0.$ Then:
\begin{itemize}
    \item[i)] If $N\equiv 0 \mod 4,$ then there exist coefficients $\{c_k^0\}_{k=0}^{\frac{N-6}{2}}$ such that
    \[
    u=\sum_{k=0}^{\frac{N-6}{2}} c_k^0 S_k^{\infty,0}
    \]
    for all $x$ with $\{\vert x\vert > R+\vert t\vert\}.$
    \item[ii)]If $N\equiv 2 \mod 4,$ then there exist coefficients $\{c_k^1\}_{k=0}^{\frac{N-8}{2}}$ such that
    \[
    u=\sum_{k=0}^{\frac{N-8}{2}} c_k^1 S_k^{\infty,1}
    \]
    for all $x$ with $\{\vert x\vert > R+\vert t\vert\}.$
\end{itemize}
If $R=0,$ then there exists a constant $c_0$ such that $u=c_0S_0^{\infty,0}$ if $N\equiv 0 \mod 4,$ and there exists a constant $c_1$ such that $u=c_1S_0^{\infty,1}$ if $N\equiv 2 \mod 4,$
\end{prop}
The analogues of these for $N=4,6$ are done in \cite{cdkm}.
\begin{proof}
First note that by replacing $u$ with $\frac{1}{2}(u(t,x)-u(-t,x))$ if $N\equiv 0 \mod 4$ or $\frac{1}{2}(u(t,x)+u(-t,x))$ if $N\equiv 2 \mod 4,$ we can assume that $u_0\equiv 0$ if $N= 0 \mod 4$ or $u_1\equiv0$ if $N= 2 \mod 4.$ Fix a large radius $R'>0.$ If $N\equiv 0 \mod 4,$ let $\{c_k^0\}_{k=0}^{\frac{N-6}{2}}$ be such that
\[
\int_{R'}^\infty \left(u_1 - \sum_{k=0}^{\frac{N-6}{2}} c_k^0 T_k^{\infty}\right) \frac{1}{r^{N-2j}} r^{N-1}dr=0,
\]
for all $1\leq j\leq \frac{N}{4},$ while in the case $N\equiv 2 \mod 4$ we let $\{c_k^1\}_{k=0}^{\frac{N-8}{2}}$ be such that
\[
\int_{R'}^\infty \left(u_0 - \sum_{k=0}^{\frac{N-8}{2}} c_k^1 T_k^{\infty}\right) \frac{1}{r^{N-2j}} r^{N-1}dr=0,
\]
for all $0\leq j\leq \frac{N-2}{4}.$ To show the existence of these coefficients, note that their existence is equivalent to the invertibility of the matrices $A=(a_{i,j})_{0\leq i,j \leq \frac{N-6}{2}}$ (or $0\leq i,j \leq \frac{N-8}{2}$ if $N \equiv 2 \mod 4$) with
\[
a_{i,j} = \int_{R'}^\infty T_i^\infty \frac{1}{r^{N-2j}} r^{N-1}dr.
\]
This matrix can be written as
\[
A=A_0(R')+A_{\mathrm{error}}(R'),
\]
where $A_0(R')=(a_{ij}^0(R'))_{0\leq i,j\leq \frac{N-6}{2}}$ (or $(a_{ij}^0(R'))_{0\leq i,j\leq \frac{N-8}{2}}$ if $N \equiv 2 \mod 4$) with
\[
a_{i,j}^0(R') = \int_{R'}^\infty \frac{1}{r^{N-2i}} \frac{1}{r^{N-2j}} r^{N-1}dr
\]
and $A_{\mathrm{error}}(R')$ is a smaller order term. More precisely, we have by explicit computation
\[
\norm{A_{\mathrm{error}}(R')}_2 = O(R'^{-1}\norm{A_0(R')}_2)
\]
where $\norm{\cdot}_2$ denotes the $L^2$ norm for matrices. The matrix $A_0(R')$ is easily seen to be invertible, since if $v=(v_i)_{0\leq i \leq \frac{N-6}{2}}$ lies in its kernel, then we obtain
\[
\int_{R'}^\infty \left(\sum_{i=0}^{\frac{N-6}{2}}v_i \frac{1}{r^{N-2i}} r^{N-1}dr\right)v_k = 0
\]
for all $1\leq k \leq \frac{N-6}{2},$ and so $v=0.$ In the case that $N \equiv 2 \mod 4,$ we set $c_{(N-6)/2}^1=0.$ Also, let $\sigma:=0$ if $N\equiv 0 \mod 4$ and $\sigma:=1$ if $N\equiv 2 \mod 4.$
Recall that by Proposition \ref{2.3} we have
\[
\norm{u-\sum_{k=0}^{\frac{N-6}{2}}c_k^\sigma S_k^{\infty,\sigma}}_{\mathcal{A},N,R'} \lesssim \norm{V\left(u-\sum_{k=0}^{\frac{N-6}{2}}c_k^\sigma S_k^{\infty,\sigma}\right)}_{L_t^1L_r^2(\mathcal{C}_{0,R'})}.
\]
Now by Lemma \ref{extstrichartz}, we have
\begin{eqnarray*}
\norm{u-\sum_{k=0}^{\frac{N-6}{2}}c_k^\sigma S_k^{\infty,\sigma}}_{L^2L^{\frac{2N}{N-3}}(\vert x\vert \geq R'+\vert t\vert)} \lesssim \norm{V\left(u-\sum_{k=0}^{\frac{N-6}{2}}c_k^\sigma S_k^{\infty,\sigma}\right)}_{L_t^1L_r^2(\mathcal{C}_{0,R'})} + \norm{u-\sum_{k=0}^{\frac{N-6}{2}}c_k^\sigma S_k^{\infty,\sigma}}_{\mathcal{A},N,R'}.
\end{eqnarray*}
Combining these two estimates and using H\"older, we obtain
\[
\norm{u-\sum_{k=0}^{\frac{N-6}{2}}c_k^\sigma S_k^{\infty,\sigma}}_{L^2L^{\frac{2N}{N-3}}(\vert x\vert \geq R'+\vert t\vert)} \lesssim \norm{V}_{L^2L^{\frac{2N}{3}}(\vert x\vert \geq R'+\vert t\vert)}\norm{u-\sum_{k=0}^{\frac{N-6}{2}}c_k^\sigma S_k^{\infty,\sigma}}_{L^2L^{\frac{2N}{N-3}}(\vert x\vert \geq R'+\vert t\vert)}.
\]
Since $V \in L^2_tL^\frac{2N}{3}_x,$ taking $R'$ large, we obtain the result for $|x|\geq R'+|t|.$ Arguing as in \cite[Lemma 3.2]{cdkm}, we obtain it for all $R>0.$ For the case of $R=0,$ note that $S_{k}^{\infty,0},S_{k}^{\infty,1} \notin E_{0,R}$ if $R=0,$ $k \geq 1.$ This shows that the only choice for $u$ is given by $c_0S_0^{\infty,0}$ or $c_0S_0^{\infty,1}.$
\end{proof}
The next step is to show a quantitative version of Proposition \ref{rigidity}. For this, let $\chi_0$ be a cutoff such that
\[
\chi_0(r) := 
\left\{
\begin{matrix}
1 & \mbox{if } r\leq 10,\\
0 & \mbox{if } r \geq 11
\end{matrix}
\right.
\]
with $\chi_0(r) \in [0,1].$
suppose that $N\geq 8.$ We define the projection $\Pi^\perp_R$ by
\begin{equation}\label{proj1}
\Pi^\perp_R (u_0,u_1) =
\left\{
\begin{matrix}
\Pi_{\dot{H}^1_R}(\mathrm{Span}\{\Lambda W\})^\perp(u_0) & \mbox{for }R=0\\
\Pi_{\dot{H}^1_R}(\mathrm{Span}\{T_k^\infty\}_{k=0}^{\frac{N-6}{2}}\cup\{\chi_0 T_k^0\}_{k=0}^{\frac{N-6}{2}})^\perp(u_0) & \mbox{for }R\in (0,1)\\
\Pi_{\dot{H}^1_R}(\mathrm{Span}\{T_k^\infty\}_{k=0}^{\frac{N-6}{2}})^\perp(u_0) & \mbox{ for } R \geq 1.
\end{matrix}
\right.
\end{equation}
in the case that $N \equiv 0 \mod 4,$ and
\begin{equation}\label{proj2}
\Pi^\perp_R (u_0,u_1) =
\left\{
\begin{matrix}
\Pi_{L^2_R}(\mathrm{Span}\{T_k^\infty\}_{k=0}^{\frac{N-8}{2}})^\perp(u_1) & \mbox{for }R=0 \mbox{ or } R \geq 1\\
\Pi_{L^2_R}(\mathrm{Span}\{T_k^\infty,T_k^0\}_{k=0}^{\frac{N-8}{2}})^\perp(u_1) & \mbox{for }R\in (0,1).
\end{matrix}
\right.
\end{equation}
in the case that $N \equiv 2 \mod 4.$
Let $\overrightarrow{T}_k^\infty, \overrightarrow{T}_k^0$ be defined as
\[
\overrightarrow{T}_k^\infty =
\left\{
\begin{matrix}
(T_k^\infty,0) & \mbox{ if } N \equiv 0\mod 4,\\
(0,T_k^\infty) & \mbox{ if } N \equiv 2\mod 4
\end{matrix}
\right.
\]
and
\[
\overrightarrow{T}_k^0 =
\left\{
\begin{matrix}
(T_k^0,0) & \mbox{ if } N \equiv 0\mod 4,\\
(0,T_k^0) & \mbox{ if } N \equiv 2\mod 4.
\end{matrix}
\right.
\]
For convenience we set $T_{\frac{N-6}{2}}^\infty,T_{\frac{N-6}{2}}^0 = 0$ in the case that $N \equiv 2 \mod 4.$
\begin{lemma}\label{notinterestingsingle}
There is a constant $C$ such that for all $R>0$ and any radial solution $u_L$ of \eqref{LW} with $u_L \in \mathcal{H}_R$ then
\[
\norm{\Pi^\perp_R (u_0,u_1)}_{\dot{H}^1}^2 \leq C \sum_{\pm}\lim_{t \to \pm \infty}\int_{r\geq R+\vert t\vert} \vert \nabla_{t,x}u_L\vert^2 dx
\]
if $N \equiv 0 \mod 4,$
\[
\norm{\Pi^\perp_R (u_0,u_1)}_{L^2} \leq C \sum_{\pm} \lim_{t \to \pm \infty} \int_{r \geq R+\vert t\vert} \vert \nabla_{t,x} u_L\vert^2 dx
\]
if $N \equiv 2 \mod 4.$
\end{lemma}
\begin{proof}
Suppose by contradiction that the result does not hold. Then there exists initial conditions $u_n$ with corresponding solution to \eqref{LW} $\{u_{L,n}\}_n$, where $u_n=(u_{0,n},0)$ if $N\equiv 0 \mod N,$ and $u_n=(0,u_{1,n})$ if $N\equiv 2 \mod N,$ with
\[
\norm{\Pi^\perp_R (u_{0,n},u_{1,n})}_{\dot{H}^1}=1
\]
if $N \equiv 0 \mod 4,$
\[
\norm{\Pi^\perp_R (u_{0,n},u_{1,n})}_{L^2}=1
\]
if $N \equiv 2 \mod 4,$ and
\begin{equation}\label{eq10}
\lim_{n\to\infty} \left(\lim_{t \to \infty}\int_{\vert x\vert \geq R_n+\vert t\vert}\vert \nabla_{t,x} u_{L,n}(t,x)\vert^2dx\right)=0.
\end{equation}
We can decompose the initial condition $u_n$ in $\{\vert x\vert \geq R_n\}$ as
\[
u_n=\sum_{k=0}^{\frac{N-6}{2}}c_{k,n} \overrightarrow{T}_k^\infty + v_n,
\]
where $v_n$ is chosen so that $v_n = (v_{0,n},0)$ if $N \equiv 0 \mod 4,$ $v_n = (0,v_{1,n})$ if $N \equiv 2 \mod 4,$ with
\begin{equation}\label{orthog}
v_{0,n} \perp_{\dot{H}^1_{R_n}} \overrightarrow{T}_k^\infty \mbox{ if } N \equiv 0 \mod 4, \; v_{1,n} \perp_{L^2_{R_n}} \overrightarrow{T}_k^\infty \mbox{ if } N \equiv 2 \mod 4
\end{equation}
and where $v_{0,n}, v_{1,n}$ have support in the exterior of $\{\vert x\vert \geq R_n\}.$ We split into cases. Let $\mathcal{A}_{N,R}$ denote the space $\dot{H}^1_R$ if $N \equiv 0 \mod 4,$ and the space $L^2_R$ if $N \equiv 2 \mod 4.$\\

\underline{Case 1:} For all $n,$ $R_n \geq 1.$ By taking a subsequence if necessary, we can additionally assume that $R_n \to R_\infty,$ where $R_\infty \geq 1.$ Note that in this case by the definition of the projection $\Pi_R^\perp,$
\[
\norm{v_n}_{\dot{H}^1_{R_n} \times L^2_{R_n}} = 1.
\]
We claim that
\[
v_n\rightharpoonup0
\]
weakly in $\dot{H}^1 \times L^2.$ Indeed, first note that if $R_\infty = \infty$ then this is trivial since $v_n(r)=0$ if $r \leq R_n.$ If not, let $v^*$ be any weak subsequential limit. Then by the orthogonality conditions \eqref{orthog} we have that
\[
v^* \perp_{\dot{H}^1 \times L^2} \overrightarrow{T}_k^\infty
\]
by weak convergence and the fact that $\norm{T_k^\infty}_{\mathcal{A}_{N,1}}<\infty.$ On the other hand, let $v_{L}$ denote the solution to \eqref{LW} with initial conditions $v^*.$ We also let $v_n^R$ be the solution to \eqref{LW} in $\{\vert x\vert \geq R\}$ with initial data $1_{r\geq R}v_n,$ where $R$ is chosen so that $R > R_\infty,$ and similarly $v^R$ be the solution to \eqref{LW} in $\{\vert x\vert \geq R\}$ with initial data $1_{r\geq R}v^*.$ Using our contradiction hypothesis and finite propagation speed we have
\begin{equation}\label{eq9bis}
\lim_{n\to\infty}\lim_{t\to\infty} \int_{r \geq R+\vert t\vert} \vert 
\nabla_{t,x}v_{L,n}^R\vert^2= 0.
\end{equation}
Since $(v_{L,n}^R(0),\p_t v_{L,n}^R(0)) \rightharpoonup v^*$ weakly in $\dot{H}^1\times L^2,$ we have that $v^*$ is in the $\dot{H}^1\times L^2$ closure of the convex hull of $\{(v_{L,n}^R(0),\p_t v_{L,n}^R(0))\}_{n \geq N_0}$ for any large $N_0.$ Hence for any $\e>0,$ there exist coefficients $\{c_k\}_{k=1}^\infty$ such that all but finitely many are $0,$ their sum is $1,$ and such that
\[
\norm{v^* - \sum_{k=0}^\infty c_k \left(v_{L,k}^R(0),\p_t v_{L,k}^R(0)\right)}_{\dot{H}^1_R\times L^2_R} < \e.
\]
Let $w$ denote the solution to \eqref{LW} with initial condition $(v^* - \sum_{k=0}^\infty c_k \left(v_{L,k}^R(0),\p_t v_{L,k}^R(0)\right))1_{\vert x\vert > R}.$ Then by Lemma 2.1 in \cite{exterior10} we have that $\sup_{t \in \R} \norm{(w,\p_t w)}_{\dot{H}^1\times L^2(|x|\geq R+|t|)} \leq C.$ This implies that $v^R$ is then in the $L^\infty_t(\dot{H}^1_R\times L^2_R)$ closure of the convex hull of $\{v_n^R\}_{n \geq N_0}.$ Thus combining this with \eqref{eq9bis} we obtain that
\begin{equation}\label{eq11}
\lim_{\vert t\vert \to \infty} \int_{r\geq R+\vert t\vert} \vert \nabla_{t,x} v_L^R\vert^2 = 0.
\end{equation}
Now using Proposition \ref{rigidity} we see that there exist coefficients $c_k$ such that $v_L^R = \sum_{k=1}^{\frac{N-6}{2}} c_k(R) \overrightarrow{S}_k^\infty$ (or $v_L^R = \sum_{k=1}^{\frac{N-8}{2}} c_k(R) \overrightarrow{S}_k^\infty$ if $N\equiv 2 \mod 4$) Since $v_L^{R_1}(0),$ $v_L^{R_2}(0)$ coincide for $\vert x\vert \geq \max\{R_1,R_2\},$ the coefficient $c_k(R)$ is independent of $R,$ and so we see that $v_L^R=0$ and therefore $v^*=0.$ Now the claim has been established.

We let $v_{F,n}$ be the free wave $(\p_{tt}-\Delta)v_{F,n}=0$ where $v_{F,n}$ has initial conditions $v_n.$ Then by Lemma 3.7 in \cite{exterior10} together with the contradiction hypothesis we have
\[
\lim_{t\to \infty}\int_{\vert x\vert \geq R_n+\vert t\vert} \vert \nabla_{t,x}v_{F,n}\vert^2 dx\geq C\norm{v_n}_{\dot{H}^1\times L^2}^2 - \sum_{j=0}^{\frac{N-4}{4}}\left(\norm{r^{-(N-2j-2)}}_{\dot{H}^1_{R_n}}\int_{\vert x\vert \geq R_n} \nabla \left(\vert x\vert^{-(N-2j-2)}\right) \cdot \nabla v_n(x) dx \right)^2.
\]
if $N \equiv 0 \mod 4,$ and
\[
\lim_{t\to \infty}\int_{\vert x\vert \geq R_n+\vert t\vert} \vert \nabla_{t,x}v_{F,n}\vert^2 dx\geq C\norm{v_n}_{\dot{H}^1\times L^2}^2 - \sum_{j=0}^{\frac{N-8}{4}}\left(\norm{r^{-(N-2j-2)}}_{L^2_{R_n}}\int_{\vert x\vert \geq R_n} \vert x\vert^{-(N-2j-2)} v_n(x) dx \right)^2.
\]
if $N \equiv 2 \mod 4.$
Now since $v_n\rightharpoonup v^*=0$ in $\dot{H}^1 \times L^2,$ we have that
\[
\lim_{n \to \infty} \norm{r^{-(N-2j-2)}}_{\dot{H}^1_{R_n}}\int_{\vert x\vert \geq R_n} \nabla \left(\vert x\vert^{-(N-2j-2)}\right) \cdot \nabla v_n(x) dx = 0
\]
if $N \equiv 0\mod 4,$ and
\[
\lim_{n \to \infty} \norm{r^{-(N-2j-2)}}_{\dot{H}^1_{R_n}}\int_{\vert x\vert \geq R_n} \vert x\vert^{-(N-2j-2)} v_n(x) dx = 0
\]
if $N \equiv 2\mod 4.$ Combining with \eqref{eq11} this yields a contradiction.
\\

\underline{Case 2:} For all $n,$ $R_n =0.$ This case is simlar to the previous one. We let $v_n = \Pi_{R}^\perp u_n,$ and so in particular
\[
\norm{v_n}_{\dot{H}^1\times L^2} \leq 1.
\]
We again claim that $v_n \rightharpoonup 0$ in $\dot{H}^1\times L^2.$ Let $v^*$ be any subsequential weak limit in $\dot{H}^1\times L^2.$ Then we have that
\[
v^* \perp_{\dot{H}^1 \times L^2} \overrightarrow{T}_0^\infty
\]
by weak convergence and the fact that $\norm{T_k^\infty}_{\mathcal{A}_N}<\infty.$ On the other hand, let $v_{L}$ denote the solution to \eqref{LW} with initial conditions $v^*.$ We also let $v_n^R$ be the solution to \eqref{LW} in $\{\vert x\vert \geq R+|t|\}$ with initial data $1_{r\geq R}v_n,$ where $R$ is chosen so that $R > R_\infty,$ and similarly $v^R$ be the solution to \eqref{LW} in $\{\vert x\vert \geq R\}$ with initial data $1_{r\geq R}v^*.$ Using our contradiction hypothesis and finite propagation speed we have
\begin{equation}\label{eq9}
\lim_{n\to\infty}\lim_{t\to\infty} \int_{r \geq R+\vert t\vert} \vert 
\nabla_{t,x}v_{L,n}^R\vert^2= 0.
\end{equation}
Since $(v_{L,n}^R(0),\p_t v_{L,n}^R(0)) \rightharpoonup v^*$ weakly in $\dot{H}^1\times L^2,$ we have that $v^*$ is in the $\dot{H}^1\times L^2$ closure of the convex hull of $\{(v_{L,n}^R(0),\p_t v_{L,n}^R(0))\}_{n \geq N_0}$ for any large $N_0.$ Hence for any $\e>0,$ there exist coefficients $\{c_k\}_{k=1}^\infty$ such that all but finitely many are $0,$ their sum is $1,$ and such that
\[
\norm{v^* - \sum_{k=0}^\infty c_k \left(v_{L,k}^R(0),\p_tv_{L,k}^R(0)\right)}_{\dot{H}^1_R\times L^2_R} < \e.
\]
Let $w$ denote the solution in $\{r>R\}\times (-\infty,\infty)$ to \eqref{LW} in $\{\vert x\vert \geq R\}$ with initial condition $(v^* - \sum_{k=0}^\infty c_k \left(v_{L,k}^R(0),\p_t v_{L,k}^R(0)\right))1_{\vert x\vert > R}.$ Then by Lemma 2.1 in \cite{exterior10} we have that $\sup_{t \in\R} \norm{(w,\p_t w)}_{\dot{H}^1\times L^2} < C.$ This implies that $v^R$ is then in the $L^\infty_t (\dot{H}^1\times L^2)$ closure of the convex hull of $\{v_n^R\}_{n \geq N_0}.$ Thus combining this with \eqref{eq9} we obtain that
\[
\lim_{\vert t\vert\to \infty} \int_{r\geq R+\vert t\vert} \vert \nabla_{t,x} v_L^R\vert^2 = 0.
\]
Now using Proposition \ref{rigidity} we see that there exists a coefficient $c=c(R)$ such that $v_L^R = c(R)\overrightarrow{S}_0^{\infty,\sigma}$ where $\sigma=0$ if $N\equiv 0 \mod 4$ and $\sigma=1$ if $N \equiv 2 \mod 4$ and where $\overrightarrow{S}_0^{\infty,\sigma}$ denotes the $\sigma$-th coordinate of $\overrightarrow{S}_0^\infty.$ Since $v_L^{R_1}(0),$ $v_L^{R_2}(0)$ coincide for $\vert x\vert \geq \max\{R_1,R_2\},$ the coefficient $c(R)$ is independent of $R$ and so we see that $v_L^R=0$ and therefore $v^* = 0.$ Now the claim has been established.

We let $v_{F,n}$ be the free wave $(\p_{tt}-\Delta)v_{F,n}=0$ where $\overrightarrow{v}_{F,n}$ has initial conditions $v_n.$ Then by Lemma 3.7 in \cite{exterior10} together with the contradiction hypothesis we have
\[
\lim_{n\to\infty}\lim_{t\to \infty}\int_{\vert x\vert \geq R_n+\vert t\vert} \vert \nabla_{t,x}v_{F,n}\vert^2 dx\geq C\norm{v_n}_{\dot{H}^1\times L^2}^2.
\]
Since $v_n\rightharpoonup v^*=0$ in $\dot{H}^1 \times L^2$ this yields a contradiction.
\\

\underline{Case 3:} For all $n,$ $R_n \in (0,1)$ and $\lim_{n \to \infty} R_n = R_\infty \in(0,1].$  Suppose $R_n \in (0,1)$ and that $\lim_{n \to \infty} R_n = R_\infty \in (0,1].$ By \eqref{eq10}, we have that $\Pi_{R_n}^\perp u_n = \Pi_{R_n}^\perp v_n,$ and so $\norm{v_n}_{\dot{H}^1_{R_n}\times L^2_{R_n}} \geq 1.$ Let
\[
\tilde{u}_n = \frac{u_n}{\norm{v_n}_{\dot{H}^1_{R_n}\times L^2_{R_n}}},\; \tilde{v}_n = \frac{v_n}{\norm{v_n}_{\dot{H}^1_{R_n}\times L^2_{R_n}}}
\]
and similarly let $\tilde{u}_{L,n}$ be the solution to \eqref{LW} with initial condition $\tilde{u}_n.$ Then note that
\[
\tilde{u}_n = \sum_{k=0}^{\frac{N-6}{2}} \tilde{c}_{k,n} T_k^\infty + \tilde{v}_n
\]
with $\tilde{c}_{k,n} = \frac{c_{k,n}}{\norm{v_n}_{\dot{H}^1_{R_n}\times L^2_{R_n}}},$ $\tilde{v}_n \perp_{\dot{H}^1\times L^2} \overrightarrow{T}_k^\infty,$ $\norm{v_n}_{\dot{H}^1_{R_n}\times L^2_{R_n}}=1,$ $R_n \to R_\infty >0,$ and
\[
\lim_{n \to \infty} \lim_{t\to \infty} \int_{\vert x\vert > R_n+\vert t\vert}\vert \nabla_{t,x} \tilde{u}_{L,n}(t,x)\vert^2 dx = 0.
\]
Now we continue as in case 1.\\

\underline{Case 4:} For all $n,$ $R_n \in (0,1)$ and $\lim_{n \to \infty} R_n = 0.$
The proof is very similar to that of case 3 in the proof of Lemma 3.1 in \cite{cdkm}. One can bound the asymptotic exterior energy of the solutions to \eqref{LW} with initial conditions given by $(T_k^0,0)$ and $(0,T_k^0).$ Then one can expand $u_n$ as
\[
u_n= \sum_{k=0}^{\frac{N-6}{2}}c_{k,n}^\infty \overrightarrow{T}_k^\infty + c_{k,n}^0 \chi_0\overrightarrow{T}_k^0 + \Pi^\perp u_n
\]
\[
u_n = \sum_{k=0}^{\frac{N-6}{2}} \tilde{c}_{k,n}^\infty \overrightarrow{T}_k^\infty +v_n.
\]
Let $\tilde{S}_k^0$ denote the solution to \eqref{LW} with initial condition $\chi_0 \overrightarrow{T}_k^0.$ Let $R>0.$ Since $\chi_0 \overrightarrow{T}_k^0 \notin \mathrm{Span}(\{T_k^\infty\}_{k=0}^{\frac{N-6}{2}}),$ by Proposition \ref{rigidity} we have that there exists a constant $C$ such that
\[
 \frac{1}{C} \leq \lim_{t \to \pm\infty} \int_{\vert x\vert \geq R+\vert t\vert} \vert \nabla \tilde{S}_k^0\vert^2 dx.
\]
On the other hand, letting $\Xi_k := \tilde{S}_k^0 - \chi_0 S_k^{0,\sigma}$ where $\sigma = 0$ if $N \equiv 0\mod 4$ and $\sigma =1$ if $N \equiv 2 \mod 4,$ we have
\[
\left\{
\begin{matrix}
(\p_t^2-\Delta + V)\Xi_K = -F_K\\
(\Xi_K,\p_t \Xi_K)\vert_{t=0} = (0,0),
\end{matrix}
\right.
\]
where
\[
F_K:= \chi_0 \left(\sum_{k=1}^K \frac{t^{2(k-1)}}{(2(k-1))!}T_{k-i}\right) + (-\Delta+V)\left(\chi_0\left(\sum_{k=0}^K \frac{t^{2i}}{(2i)!} T_{k-i}\right)\right)
\]
which is $0$ if $r \leq 10$ or $r \geq 12.$ From this one obtains as in the proof of \eqref{eq11}, with standard energy estimates that there is a constant $C$ such that
\begin{equation}\label{eq12}
\frac{1}{C} \leq \lim_{t \to \infty}\int_{\vert x\vert \geq R+\vert t\vert} \vert \nabla \tilde{S}_K^0\vert^2 dx\leq C.
\end{equation}
We take the exterior energy on both sides to obtain
\begin{eqnarray*}
\sum_{k_1,k_2=0}^{\frac{N-6}{2}} a_{k_1k_2} c_{k_1,n}^0 c_{k_2,n}^0 &=& \lim_{t \to \infty} \int_{\vert x\vert > R+\vert t\vert} \left\vert \nabla_{t,x} \sum_{k=0}^{\frac{N-6}{2}} c_{k,n}^0 \tilde{S}_k^0\right\vert^2 dx\\
&\lesssim & \lim_{t\to \infty} \int_{\vert x\vert > R+\vert t\vert} \vert \nabla_{t,x} u_{L,n}\vert^2 + \vert \nabla_{t,x}\bar{w}_{L,n}\vert^2 + \left\vert \nabla_{t,x} \sum_{k=0}^{\frac{N-6}{2}} c_{k,n}^\infty c_{k,n}^\infty S_k^\infty\right\vert^2 dx
\end{eqnarray*}
which is uniformly bounded in $n,$
where
\[
a_{k_1k_2} := \lim_{t\to \pm\infty} \int_{\vert x\vert > R+\vert t\vert} \nabla_{t,x} \tilde{S}_{k_1}^0 \cdot \nabla_{t,x} \tilde{S}_{k_2}^0 dx.
\]
Note that the matrix $A:=(a_{k_1k_2})_{0 \leq k_1,k_2 \leq (N-6)/2}$ is symmetric, positive semi-definite. We claim that it is also invertible, hence strictly positive definite. To see this, suppose that $v = (v_k)_{0\leq k \leq N-6/2}$ lies in its kernel. Letting $f := \sum_{k=0}^{\frac{N-6}{2}} v_k \tilde{S}_k^0$ we obtain
\begin{eqnarray*}
0&=& \sum_{i=0}^{\frac{N-6}{2}} v_i \lim_{t\to\infty} \int_{\vert x\vert > R=\vert t\vert} \nabla_{t,x}\tilde{S}_k^0 \cdot \nabla_{t,x}\tilde{S}_i^0\\
&=& \lim_{t \to \infty} \int_{\vert x\vert > R=\vert t\vert} \nabla_{t,x} \tilde{S}_k^0 \cdot \nabla_{t,x} f.
\end{eqnarray*}
Multiplying the above identity by $v_k$ and adding over all $k$ we obtain that
\[
\lim_{t\to\infty} \int_{\vert x\vert > R+\vert t\vert} \vert \nabla_{t,x} f\vert^2 = 0.
\]
By Lemma \ref{rigidity} we have that $f \in \mathrm{Span}\{S_k^\infty\}_{0\leq k \leq N-6/2}.$ Thus $f \in \mathrm{Span}\{S_k^\infty\}_{0\leq k \leq N-6/2} \cap \mathrm{Span}\{S_k^0\}_{0\leq k \leq N-6/2}.$ This intersection is $\{0\}$ as can be seen by successively taking $(-\Delta+V)$ on an equality of the form $\sum_{k=0}^{\frac{N-6}{2}} d_k T_k^\infty = \sum_{k=0}^{\frac{N-6}{2}} d_k T_k^0.$ Therefore $A$ is invertible.

From this we conclude that $c_{k,n}^0$ is bounded. Now we have
\[
\sum_{k=0}^{\frac{N-6}{2}} (c_{k,n}^\infty - \tilde{c}_{k,n}^\infty) \overrightarrow{T}_k^\infty = \sum_{k=0}^{\frac{N-6}{2}} c_{k,n}^0 (\tilde{S}_k^0(0),\p_t \tilde{S}_k^0(0)) + \Pi^\perp u_n - v_n.
\]
Therefore taking the $\dot{H}^1\times L^2$ inner product with $\{(\tilde{S}_k^0(0),\p_t \tilde{S}_k^0(0))\}_{k=0}^{\frac{N-6}{2}},$ we obtain
\[
B\overrightarrow{c} = \overrightarrow{b},
\]
where
\[
\overrightarrow{c} = (c_{k,n}^\infty-\tilde{c}_{k,n}^\infty)_{k=0}^{\frac{N-6}{2}},
\]
\[
\overrightarrow{b} = \left(\sum_{i=0}^{\frac{N-6}{2}}c_{k,n}^0 \langle (\tilde{S}_i^0(0),\p_t \tilde{S}_i^0(0)),\overrightarrow{T}_k^\infty\rangle_{\mathcal{H}_{R_n}}\right)_{k=0}^{\frac{N-6}{2}} + (\langle (\Pi^\perp u_n - v_n), \tilde{T}_k^\infty\rangle_{\dot{H}^1\times L^2})_{k=0}^{\frac{N-6}{2}}
\]
and
\[
B=(b_{ij})_{0 \leq i,j\leq \frac{N-6}{2}}
\]
with
\[
b_{k_1k_2}=\langle T_{k_1}^\infty,T_{k_2}^\infty\rangle_{\mathcal{H}_{R_n}}.
\]
Since $B$ is invertible by the same argument as that for $A$ above, we obtain that $c_{k,n}-\tilde{c}_{k,n}$ is bounded since $\overrightarrow{b}$ is bounded. Thus in particular $c_{k,n}^\infty-\tilde{c}_{k,n}^\infty$ is bounded. Now we have
\[
\sum_{k=0}^{\frac{N-6}{2}} (c_{k,n}^\infty - \tilde{c}_{k,n}^\infty) \overrightarrow{T}_k^\infty = \sum_{k=0}^{\frac{N-6}{2}}c_{k,n}^0 (\tilde{S}_k^0(0),\p_t \tilde{S}_k^0(0)) + \Pi^\perp u_n - v_n.
\]
Therefore taking the $\dot{H}^1_{R_n}\times L^2_{R_n}$ inner product with $\{(\tilde{S}_k^0(0),\p_t \tilde{S}_k^0(0))\}_{k=0}^{\frac{N-6}{2}},$ we obtain
\[
B\overrightarrow{c} = \overrightarrow{b},
\]
where
\[
\overrightarrow{c} = (c_{k,n}^\infty-\tilde{c}_{k,n}^\infty)_{k=0}^{\frac{N-6}{2}},
\]
\[
\overrightarrow{b} = \sum_{i=0}^{\frac{N-6}{2}}\left(c_{i,n}^0 \langle (\tilde{S}_i^0(0),\p_t \tilde{S}_i^0(0)),\overrightarrow{T}_k^\infty\rangle_{\mathcal{A}_N}\right)_{k=0}^{\frac{N-6}{2}} + (\langle (\Pi^\perp u_n - v_n), \overrightarrow{T}_k^\infty\rangle_{\dot{H}^1\times L^2})_{k=0}^{\frac{N-6}{2}}
\]
and
\[
B=(b_{ij})_{0 \leq ij\leq \frac{N-6}{2}}
\]
with
\[
b_{k_1k_2}=\langle T_{k_1}^\infty,T_{k_2}^\infty\rangle_{\mathcal{A},R_n}.
\]
Since $B$ is invertible by the same argument as that for $A$ above, we obtain that $\overrightarrow{c} = B^{-1}\overrightarrow{b}.$ Noting that $b_{0k},b_{k0}$ for $0 \leq k \leq \frac{N-6}{2}$ are of order $O(1)$ while all other entries are of order $R_n^{-N}$ if $N \equiv 0 \mod 4,$ $R_n^{-(N-2)}$ if $N\equiv 2 \mod 4,$ we see that
\[
\mathrm{det} B =
\left\{
\begin{matrix}
O(R_n^{-N \frac{N-6}{2}}) & \mbox{ if } N \equiv 0 \mod 4\\
O(R_n^{-(N-2) \frac{N-6}{2}}) & \mbox{ if } N \equiv 2 \mod 4
\end{matrix}
\right.
\]
and so if $N \equiv 0 \mod 4$ (resp. $N \equiv 2 \mod 4$) the coordinates of $B^{-1}$ are of order $R_n^N$ (resp. $R_n^{-N+2}$), except for $b^{00}$ which is of order $O(1).$ This implies that $\overrightarrow{c}$ is bounded. Thus $v_{1,n}$ is bounded in $\dot{H}^1(r^{2N-\frac{5}{2}}\langle r\rangle^{-(N-\frac{3}{2})}dr)$ if $N \equiv 0\mod 4$ or $L^2(r^{2N-\frac{9}{2}}\langle r\rangle^{-(N-\frac{7}{2})}dr)$ if $N \equiv 2\mod 4.$ Thus in particular $c_{k,n}^\infty-\tilde{c}_{k,n}^\infty$ is bounded, and thus after proceeding as in case 1 using Lemma \ref{rigidity} one obtains that $v_n\rightharpoonup 0$ in $\dot{H}^1_R\times L^2_R$ for any $R>0,$ and so 
\[
v_n\rightharpoonup 0
\]
in $\dot{H}^1(r^{2N-\frac{5}{2}}\langle r\rangle^{-\left(N-\frac{3}{2}\right)})\times L^2(r^{2N-\frac{9}{2}}\langle r\rangle^{-\left(N-\frac{7}{2}\right)}).$ Next, we decompose $v_n$ as
\[
v_n = \sum_{i=0}^{\frac{N-6}{2}}\bar{c}_n^i (\chi_0T_i^0,0) + w_n
\]
if $N \equiv 0 \mod 4,$
\[
v_n = \sum_{i=0}^{\frac{N-6}{2}}\bar{c}_n^i (0,\chi_0T_i^0) + w_n
\]
if $N \equiv 2 \mod 4,$ where $w_n \perp_{\dot{H}^1(r^{2N-\frac{5}{2}}\langle r\rangle^{-\left(N-\frac{3}{2}\right)})\times L^2(r^{2N-\frac{9}{2}}\langle r\rangle^{-\left(N-\frac{7}{2}\right)})} (T_i^0,0)$. We extend $w_n$ trivially for $r\leq R_n.$ Note that
\[
\norm{w_n}_{\dot{H}^1_{R_n}\times L_{R_n}^2}^2 + \sum_{i=0}^{\frac{N-6}{2}}(c_{i,n}^0-\bar{c}_n^i)^2\norm{\chi_0 T_i^0}_{\dot{H}^1}^2 \lesssim \norm{\Pi^\perp u_n}_{\dot{H}^1_{R_n}\times L^2_{R_n}}^2 + \sum_{k=0}^{\frac{N-6}{2}} (c_{k,n}^\infty - \tilde{c}_{k,n}^\infty)^2 \norm{T_k^\infty}_{\dot{H}^1_{R_n}}^2
\]
if $N \equiv 0 \mod 4,$
\[
\norm{w_n}_{\dot{H}^1_{R_n}\times L_{R_n}^2}^2 + \sum_{i=0}^{\frac{N-6}{2}}(c_{i,n}^0-\bar{c}_n^i)^2\norm{\chi_0 T_i^0}_{L^2_{R_n}}^2 \lesssim \norm{\Pi^\perp u_n}_{\dot{H}^1_{R_n}\times L^2_{R_n}}^2 + \sum_{k=0}^{\frac{N-6}{2}} (c_{k,n}^\infty - \tilde{c}_{k,n}^\infty)^2 \norm{T_k^\infty}_{L^2_{R_n}}^2
\]
if $N \equiv 2 \mod 4.$ Now after proceeding as in step 3 of the proof of Lemma 3.1 in \cite{cdkm}, we obtain that $w_n$ converges to $(0,0)$ weakly in $\dot{H}^1(r^{2N-\frac{5}{2}}\langle r\rangle^{-\left(N-\frac{3}{2}\right)})\times L^2(r^{2N-\frac{9}{2}}\langle r\rangle^{-\left(N-\frac{7}{2}\right)}).$ Now we can proceed as in step 4 of the proof of Lemma 3.1 in \cite{cdkm}. \begin{comment}{Here one needs to prove/use that $(c_{k,n}^\infty - \tilde{c}_{k,n}^\infty)^2 \norm{T_k^\infty}_{L^2_{R_n}}^2$ is bounded. In fact, how to do this step even for the $8$ dimensional case is unclear. To compute the inverse of the matrix $B$ above, note that
\[
B =
\begin{pmatrix}
\langle \Lambda W,\Lambda W\rangle_{\dot{H}^1_{R_n}} & \langle \Lambda W,T_1^\infty\rangle_{\dot{H}^1_{R_n}}\\
\langle \Lambda W,T_1^\infty\rangle_{\dot{H}^1_{R_n}} & \langle \Lambda T_1^\infty,T_1^\infty\rangle_{\dot{H}^1_{R_n}}
\end{pmatrix}
\]
To compute the inverse of this matrix, one has to compute the determinant's order. Using the asymptotic expansions for $\Lambda W,T_1^\infty$ we obtain that
\[
\langle \Lambda W,\Lambda W\rangle_{\dot{H}^1_{R_n}} \sim c_1 R_n^{-6},
\]
\[
\langle \Lambda W,T_1^\infty\rangle_{\dot{H}^1_{R_n}} \sim c_2 R_n^{-4},
\]
\[
\langle T_1^\infty,T_1^\infty\rangle_{\dot{H}^1_{R_n}} \sim c_3 R_n^{-2}
\]
for some constants $c_1,c_2,c_3.$ For this computation it matters if $c_1c_3-c_2^2 = 0$ or not, which requires additional computations. It this is indeed $0,$ one needs to extract the next order terms to see the order of the determinant, requiring even more computations. Perhaps I'm missing an easier way of bounding these coefficients, I would imagine that since it is said in the paper that the case of dimension $8$ is ``identical" it does not actually require any new elements. In any case, this argument is not easily generalizable for higher dimensions.}
\end{comment}
Letting $w_{F,n}$ be the free wave with initial condition $w_n$ and $w_{L,n}$ the free wave with initial condition $w_n,$ we obtain
\[
\lim_{n\to \infty} \left(\lim_{t\to\infty} \int_{\vert x\vert \geq R_n+\vert t\vert} \vert \nabla_{t,x} w_{F,n}\vert^2\right) = 0.
\]
Moreover we have
\begin{eqnarray*}
&&\left(\norm{\frac{1}{r^{N-2k}}}_{\dot{H}^1_{R_n}}\right)^{-1}\left\vert \int_{r \geq R_n} \nabla w_{F,n} \cdot \nabla \frac{1}{r^{N-2k}} r^{N-1}dr \right\vert\\
&&= \left(\norm{\frac{1}{r^{N-2k}}}_{\dot{H}^1_{R_n}}\right)^{-1}\left\vert \int_{r \geq R_n} \nabla w_{F,n} \cdot \nabla \left(\frac{1}{r^{N-2k}}-\chi_0 T_k^0\right) r^{N-1}dr \right\vert\\
&&\lesssim \left(\norm{\frac{1}{r^{N-2k}}}_{\dot{H}^1_{R_n}}\right)^{-1}\norm{\frac{1}{r^{N-2k}}-\chi_0 T_k^0}_{\dot{H}^1_{R_n}}\to 0
\end{eqnarray*}
if $N \equiv 0 \mod 4$ and similarly for the case $N \equiv 2 \mod 4.$ However,
\[
\norm{w_n}_{\dot{H}^1_{R_n}\times L^2_{R_n}} \geq \norm{\Pi^\perp w_n}_{\dot{H}^1_{R_n}} = \norm{\Pi^\perp u_n}_{\dot{H}^1_{R_n}} = 1
\]
if $N \equiv 0 \mod 4,$ and
\[
\norm{w_n}_{\dot{H}^1_{R_n}\times L^2_{R_n}} \geq \norm{\Pi^\perp w_n}_{L^2_{R_n}} = \norm{\Pi^\perp u_n}_{L^2_{R_n}} = 1
\]
if $N \equiv 2 \mod 4,$ which contradicts Lemma \ref{2.3}.
\end{proof}

\section{Single soliton case (resonant component)}

Let $\chi_0$ be a smooth cutoff such that
\[
\chi^0(r) =
\left\{
\begin{matrix}
1 & \mbox{ if } r \leq 10\\
0 & \mbox{ if } r \geq 11.
\end{matrix}
\right.
\]
\[
\Pi_{L_R^2}^\perp =
\left\{
\begin{matrix}
\Pi_{L_R^2}(\mathrm{span}\{T_0^\infty, \ldots T_{\frac{N-6}{2}}^\infty\})^{\perp} \mbox{ if } R \geq 1\\
\Pi_{L_R^2}(\mathrm{span}\{T_0^\infty, \ldots T_{\frac{N-6}{2}}^\infty,\chi^0 T_0^0, \ldots,\chi^0 T_{\frac{N-6}{2}}^0\})^{\perp} \mbox{ if } 0<R\leq 1,
\end{matrix}
\right.
\]
and similarly
\[
\Pi_{\dot{H}^1_R}^\perp =
\left\{
\begin{matrix}
\Pi_{\dot{H}^1_R}(\mathrm{span}\{T_0^\infty, \ldots T_{\frac{N-6}{2}}^\infty\})^{\perp} \mbox{ if } R \geq 1\\
\Pi_{\dot{H}^1_R}(\mathrm{span}\{T_0^\infty, \ldots T_{\frac{N-6}{2}}^\infty,\chi^0 T_0^0, \ldots,\chi^0 T_{\frac{N-6}{2}}^0\})^{\perp} \mbox{ if } 0<R\leq 1.
\end{matrix}
\right.
\]

Let
\[
\norm{f}_{Z_{\alpha}}=\sup_{R>0} \frac{R^{-\frac{N}{2}-\alpha}}{\langle \log R\rangle} \left( \int_{R < \vert x\vert <2R} f^2(x)dx\right)^{\frac{1}{2}},
\]
\[
\norm{f}_{Z_{\alpha,R}}=\sup_{\rho\geq R} \frac{\rho^{-\frac{N}{2}-\alpha}}{\langle \log \frac{\rho}{\langle R\rangle}\rangle} \left( \int_{\rho < \vert x\vert <2\rho} f^2(x)dx\right)^{\frac{1}{2}},
\]
and
\[
d_{Z_{\alpha,R}}(u,E):= \inf_{v \in E} \norm{u-v}_{Z_{\alpha,R}}.
\]
We will prove the following.
\begin{lemma}\label{otherpart}
Suppose $u_L$ is a radial solution to \eqref{LW}. Then
\[
\left.
\begin{matrix}
\norm{\Pi_{L^2}^\perp u_1}_{Z_{-4}}^2 & \mbox{ if } N \equiv 0 \mod 4\\
\norm{\nabla \Pi_{\dot{H}^1}^{\perp} u_0}_{Z_{-3}}^2& \mbox{ if }N\equiv 2 \mod 4
\end{matrix}
\right\}
\leq C \sum_{\pm}\lim_{t \to \pm \infty} \int_{r > \vert t\vert} \vert \nabla_{t,x} u_L(t,x)\vert^2 dx.
\]
\end{lemma}

We will prove the cases $N \equiv 0 \mod 4$ and $N \equiv 2\mod 4$ separately, in subsection \ref{other0} and \ref{other2}

\subsection{The $N\equiv 0 \mod 4$ case}\label{other0}

Let
\[
\norm{u_L}_{\tilde{Y}}:=\sup_{t\in \R,R>\vert t\vert} \norm{\Pi_{\dot{H}^1_R}^\perp u_L(t)}_{\dot{H}_R^1}
\]
and similarly let
\[
\norm{u_L}_{\tilde{Y}_R}:=\sup_{t\in \R,\rho>\vert t\vert+R} \norm{\Pi_{\dot{H}^1_\rho}^\perp u_L(t)}_{\dot{H}_\rho^1}.
\]
We also let
\[
\mathcal{A}f(r) := \int_r^\infty \rho f(\rho) d\rho, \; \mathcal{A}^{-1} = -\frac{1}{r}\p_r f(r).
\]
We will need the following lemma.
\begin{lemma}\label{56equiv}
Suppose that $u_L$ is a linear radial solution to \eqref{LW}. Let $\Pi_{L^2}^\perp$ be defined as in \eqref{Pi}. Suppose that $u_L$ is odd in time. Then if $N \equiv 0 \mod 4$ and $N \geq 10,$ then
\[
\norm{\mathcal{A}\Pi_{L^2}^\perp \p_t u_L(0)}_{Z_{-2}}\lesssim\norm{u_L}_{\tilde{Y}}.
\]
\end{lemma}
Note that the above lemma is proven in the case of $N=8$ in \cite{cdkm}.
\begin{proof}
First, note that without loss of generality, we can assume that
\begin{equation}\label{assump1}
\Pi_{L^2}^\perp u_L(0) = u_L(0).
\end{equation}
Indeed, suppose we have showed the result for solutions satisfying \eqref{assump1}. Then if $u_L$ is any solutions satisfying the hypotheses of the lemma, we let $\bar{u}_L$ denote the solution to \eqref{LW} with initial condition given by $(0,\Pi_{L^2}^\perp \p_t u_L(0)),$ and we let $v := u_L- \bar{u}_L.$ Then we have
\[
\norm{\mathcal{A}\Pi_{L^2}^\perp \p_t u_L(0)}_{Z_{-2}} = \norm{\mathcal{A} \p_t \bar{u}_L(0)}_{Z_{-2}} \lesssim\norm{\bar{u}_L}_{\tilde{Y}} \lesssim \norm{u_L}_{\tilde{Y}},
\]
where the last inequality comes from direct computation, using that by Lemma \ref{varofpars}
\[
\vert T_k^\infty(r)\vert + \left|r\frac{d}{dr}T_k^\infty(r)\right| \lesssim \langle r\rangle^{-N+2+2k}.
\]
Let $R_k:=2^k.$ Then note that
\begin{eqnarray}\label{ptdecomp2}
\p_{t} u_L(0)&=& - \int_0^{R_k}\p_{t}\left(\left(1-\frac{t}{R_k}\right)\p_t u_L(t)\right) dt\nonumber\\
 &=& \int_0^{R_k} \left(1-\frac{t}{R_k}\right)(-\p_{tt})u_L(t)dt + \frac{1}{R_k}\int_0^{R_k}\p_tu_L(t)dt\nonumber\\
 &=& (-\Delta + V)\int_0^{R_k} \left(1-\frac{t}{R_k}\right) u_L(t)dt + \frac{u_L(R_k)}{R_k}
\end{eqnarray}
where in the last line we used the fact that $u_L$ is odd in time, hence $ u_L(0)=0.$ Let $\tilde{u}_k(t) = \Pi_{\dot{H}_{R_k}^1}^\perp u_L(t).$ Writing $T_k^0$ in terms of $\tilde{T}_k^0$ by Lemma \ref{decompo}, there exist coefficients $\alpha_{j,k}^\infty = \alpha_{j,k}^\infty(t),\alpha_{j,k}^0 = \alpha_{j,k}^0(t)$ such that
\[
u_L(t)=\sum_{j=0}^{\frac{N-6}{2}} \alpha_{j,k}^\infty T_j^\infty + \alpha_{j,k}^0 \tilde{T}_j^0 + \tilde{u}_k(t)
\]
where by convention $\alpha_{j,k}^0 = 0$ if $k \geq 0$ and where we let $\tilde{T}_j^0$ be defined as in Lemma \ref{decompo}.
\begin{comment}
Letting
\[
-\bar{T}_{i-1}^0 = (-\Delta+V)\tilde{T}_i^0
\]
\end{comment}
This implies
\begin{eqnarray}\label{ptdecomp1}
&&(-\Delta+V)\left(\int_0^{R_k}\left(1-\frac{t}{R_k}\right)u_L(t)dt\right)\nonumber\\
&&= \int_0^{R_k} \left(1-\frac{t}{R_k}\right)\left(\left(\sum_{j=0}^{\frac{N-6}{2}} -\alpha_{j+1,k}^\infty(t) T_j^\infty - \alpha_{j+1,k}^0(t)\tilde{T}_j^0(t)\right)+(-\Delta+V)\tilde{u}_k(t)\right) dt
\end{eqnarray}
whenever $r\geq R_k$ if $k\geq 0,$ or for all $R_k \leq r \leq 10$ if $k \leq -1.$ Recall that by Lemma \ref{decompo}, there exist coefficients $e_{j,i}^\infty,e_{j,i}^0$ such that
\[
\tilde{T}_i^0 := T_i^0 - \sum_{j =1}^{i+1} e_{j,i}^\infty T_j^\infty - \sum_{j=1}^{i-1} e_{j,i}^0 \tilde{T}_j^0 \sim r^2
\]
for $r \to 0.$ Plugging this and \eqref{ptdecomp1} into \eqref{ptdecomp2}, we obtain that for some coefficients $\{c_{j,k}^\infty\},\{c_{j,k}^0\}$
\begin{eqnarray*}
\p_t u(0) = \frac{1}{R_k}\tilde{u}_k(R_k) - \sum_{j=0}^{\frac{N-6}{2}} c_{j,k}^\infty T_j^\infty - \sum_{j=0}^{\frac{N-6}{2}} c_{j,k}^0 \tilde{T}_j^0 + (-\Delta+V) \left(\int_0^{R_k}\left(1-\frac{t}{R_k}\right)\tilde{u}_k(t) dt\right).
\end{eqnarray*}
Using the above equality for $k$ and $k+1$ we obtain
\begin{eqnarray}\label{elimination}
&&\sum_{j=0}^{\frac{N-6}{2}} (c_{j,k+1}^\infty - c_{j,k}^\infty)T_j^\infty + \sum_{j=0}^{\frac{N-6}{2}} (c_{j,k+1}^0 - c_{j,k}^0)\tilde{T}_j^0\nonumber\\
&&= \frac{1}{R_{k+1}}\tilde{u}_{k+1}(R_{k+1}) - \frac{1}{R_k}\tilde{u}_k(R_k)\nonumber\\
&&+ (-\Delta+V)\left(\int_0^{R_{k+1}}\left(1-\frac{t}{R_{k+1}}\right)\tilde{u}_{k+1}(t)dt-\int_0^{R_k}\left(1-\frac{t}{R_k}\right)\tilde{u}_k(t)dt\right).
\end{eqnarray}
Now we choose functions $\Phi_k,\tilde{\Phi}_k$ such that $\mathrm{Supp} \Phi_k, \mathrm{Supp} \tilde{\Phi}_k \subseteq [2R_k,3R_k],$
\[
\vert \p_r^j \Phi_k\vert \lesssim R_k^{-j}, \; \vert \p_r^j \tilde{\Phi}_k\vert \lesssim R_k^{-j}
\]
for $j = 0, \ldots , N-6$ and such that
\[
\int_{\R^N} \tilde{T}_0^0 \tilde{\Phi}_k = R_k^2, \; \int_{\R^N} \tilde{T}_j
^0 \tilde{\Phi}_k = 0 \mbox{ if }j \geq 1,
\]
\[
\int_{\R^N} T_j^\infty \tilde{\Phi}_k = 0 \mbox{ if } k \geq -1,
\]
and
\[
\int_{\R^N} T_0^\infty \Phi_k = R_k^2, \; \int_{\R^N} T_j^\infty \Phi_k = 0 \mbox{ if } j \geq 1,
\]
\[
\int_{\R^N} T_j^0 \Phi_k = 0 \mbox{ if } k \geq -1.
\]
To show this exists, we rescale by $R_k.$ Let $\tilde{\Phi}_k(r):= \tilde{\Phi}^k\left(\frac{r}{R_k}\right).$ Then we obtain the conditions
\[
\int_0^\infty \frac{\tilde{T}_0^0(R_kr)}{R_k^{-(N-2)}} \tilde{\Phi}^k(r) r^{N-1}dr = R_k^2
\]
\[
\int_0^\infty F(R_kr) \tilde{\Phi}^k(r) r^{N-1}dr = 0
\]
for all $F \in\{\tilde{T}_j^0\}_{j=1}^{\frac{N-6}{2}} \cup \{T_j^\infty\}_{j = 0}^{\frac{N-6}{2}}.$ Since $R_k^{(N-2)}\tilde{T}_0^0(R_kr) \sim c r^{-(N-2)}$ for some absolute constant $c$ as $r\to 0,$ we can take a function $\Phi^k \perp_{L^2(r^{N-1}dr)} \mathrm{Span}\{F(R_k\cdot) \vert F\in \{\tilde{T}_j^0\}_{j \geq 1} \cup \{T_j^\infty\}_{j \geq 0}\}$ with
\[
\left\langle \Phi^k,\frac{\tilde{T}_0^0(R_k \cdot)}{R_k^{-(N-2)}}\right\rangle = 1
\]
and $\vert \p_r^j \Phi^k\vert \lesssim 1$ for all $0 \leq j \leq N-6.$ We can then take $\Phi_k(r):= \tilde{\Phi}^k\left(\frac{r}{R_k}\right).$ The existence of $\tilde{\Phi}_k$ is proven similarly.

Fix any $R>0,$ and let $k_0$ be the unique integer such that $R_{k_0-2} \leq R < R_{k_0-1}$ and let $k_1 = \max \{k_0,0\}.$ Now applying $(-\Delta + V)^j$ for $1\leq j \leq \frac{N-6}{2}$ on both sides of \eqref{elimination} and integrating against $\Phi_k$ if $k \geq -1$ and $\Phi^k$ if $k \leq -1$ we obtain
\[
\vert c_{j,k+1}^\infty - c_{j,k}^\infty\vert \leq \norm{u_L}_{\tilde{Y}_R}R_k^{-2j}
\]
for all $k \geq -1,$ and after plugging these estimates back into \eqref{elimination} we obtain
\[
\vert c_{j,k+1}^\infty - c_{j,k}^\infty\vert \lesssim \norm{u_L}_{\tilde{Y}_R} R_k^{-1-2j}
\]
for all $k \leq -2.$ Then if we write $c_{j,k}^\infty = c_{j,k_1}^\infty+\bar{c}_{j,k}^\infty,$ $c_{j,k}^0 = c_{j,k_1}^0 + \bar{c}_{j,k}^0,$
\begin{equation}\label{eq108}
\bar{c}_{j,k}^\infty = 
\left\{
\begin{matrix}
O(kR_k^{-2j}\norm{u_L}_{\tilde{Y}_R}) \mbox{ if } k \geq 0,\\
O(R_k^{-2j-1} \norm{u_L}_{\tilde{Y}_R}) \mbox{ if }k_0\leq k \leq -1
\end{matrix}
\right.
\end{equation}

if $k_0 \leq 0,$ and
\begin{equation}\label{eq109}
\bar{c}_{j,k}^\infty = O(\vert k-k_0 \vert \norm{u_L}_{\tilde{Y}_R} R_k^{-2j})
\end{equation}
if $k\geq k_0.$ Similarly, we have
\begin{equation}\label{eq110}
\vert c_{j,k+1}^0 - c_{j,k}^0\vert \lesssim R_k^{-2j} \norm{u_L}_{\tilde{Y}_R}
\end{equation}
and so
\begin{equation}\label{eq111}
\vert \bar{c}_{j,k}^0\vert \lesssim \vert k\vert R_k^{-2j}\norm{u_L}_{\tilde{Y}_R}.
\end{equation}
Therefore
\begin{eqnarray*}
\p_tu_L(0) &=& \frac{1}{R_k}\tilde{u}_k(R_k) + \sum_{j=0}^{\frac{N-6}{2}} c_{j,k}^\infty T_j^\infty + \sum_{j=0}^{\frac{N-6}{2}} c_{j,k}^0 \tilde{T}_j^0\\
&&+ (-\Delta +V)\left(\int_0^{R_k} \left(1-\frac{t}{R_k}\right)\tilde{u}_k(t)dt\right)\\
&=&\sum_{j=0}^{\frac{N-6}{2}} c_{j,k_1}^\infty T_j^\infty + \left(\sum_{j=0}^{\frac{N-6}{2}} (\bar{c}_{j,k}^\infty T_j^\infty + c_{j,k}^0 \tilde{T}_j^0)+\frac{1}{R_k}\tilde{u}_k(R_k)\right)\\
&&+ (-\Delta + V)\left(\int_0^{R_k} \left(1-\frac{t}{R_k}\right)\tilde{u}_k(t)dt\right).
\end{eqnarray*}
Let $\eta$ be such that $\eta_k:= \eta\left(\frac{r}{R_k}\right)$ is a partition of unity, that is
\[
\sum_{k=-\infty}^\infty \eta\left(\frac{r}{R_k}\right)=1.
\]
Additionally, suppose $\mathrm{Supp} \eta \subseteq [1,3],$ and let
\[
f_k^1 := \left(\sum_{j=0}^{\frac{N-6}{2}} \bar{c}_{j,k}^\infty T_j^\infty + c_{j,k}^0 \tilde{T}_j^0\right) + \frac{1}{R_k} \tilde{u}_k(R_k).
\]
Then if $k \geq k_0,$ $R_k \leq r \leq 3R_k,$ we have
\begin{eqnarray*}
\p_tu_L(0) &=&\sum_{k \geq k_0} \eta_k \p_tu_L(0)\\
&=& \sum_{k \geq k_0}\eta_k \left(\sum_{j=0}^{\frac{N-6}{2}} \bar{c}_{j,k_1}^\infty T_j^\infty + \left(\sum_{j=0}^{\frac{N-6}{2}} (\bar{c}_{j,k}^\infty T_j^\infty +c_{j,k}^0\tilde{T}_j^0) + \frac{1}{R_k}\tilde{u}_k(R_k)\right)\right.\\
&& \quad \quad \quad \quad \quad\left.+ (-\Delta + V)\left(\int_0^{R_k} \left(1-\frac{t}{R_k}\right) \tilde{u}_k(t)dt\right)\right)\\
&=& \sum_{j=0}^{\frac{N-6}{2}} c_{j,k_1}^{\infty} T_j^\infty + \sum_{k \geq k_0} \eta_k \left(\left(\sum_{j=0}^{\frac{N-6}{2}} \bar{c}_{j,k}^\infty T_j^\infty + c_{j,k}^0 \tilde{T}_j^0\right) + \frac{1}{R_k} \tilde{u}_k(R_k)\right)\\
&+& \sum_{k \geq k_0} \eta_k (-\Delta +V)\left(\int_0^{R_k} \left(1-\frac{t}{R_k}\right)\tilde{u}_k(t) dt\right)\\
&=:& \sum_{j=0}^{\frac{N-6}{2}} c_{j,k_1}^\infty T_j^\infty + \sum_{k \geq k_0}\eta_k f_k^1 + \sum_{k \geq k_0}\eta_k(-\Delta+V)\left(\int_0^{R_k}\left(1-\frac{t}{R_k}\right)\tilde{u}_k(t)dt\right).
\end{eqnarray*}
Note that
\begin{eqnarray*}
&&(-\Delta+V) \left(\eta_k\left(\int_0^{R_k} \left(1-\frac{t}{R_k}\right)\tilde{u}_k(t)dt\right)\right) - \eta_k (-\Delta + V)\left(\int_0^{R_k}\left(1-\frac{t}{R_k}\right)\tilde{u}_k(t)dt\right)\\
&&=R_k^{-2} (2\eta_k''-\Delta \eta_k) \left(\frac{r}{R_k}\right) \int_0^{R_k}\left(1-\frac{t}{R_k}\right)\tilde{u}_k(t)dt+ \p_r\left(-2R_k^{-1}\eta_k'\left(\frac{r}{R_K}\right)\int_0^{R_k}\left(1-\frac{t}{R_k}\right) \tilde{u}_k(t) dt\right)\\
&& =: f_k^2 + \p_r \tilde{g}_k.
\end{eqnarray*}
This implies that
\begin{eqnarray*}
\p_tu_L(0) &=& \sum_{j=0}^{\frac{N-6}{2}} c_{j,k_1}^\infty T_j^\infty + \sum_{k \geq k_0} (\eta_k f_k^1 + f_k^2) + \p_r\left(\sum_{k \geq k_0} \tilde{g}_k\right)\\
&+& (-\Delta +V) \left(\sum_{k \geq k_0} \eta_k \left(\int_0^{R_k} \left(1-\frac{t}{R_k}\right)\tilde{u}_k(t)dt\right)\right).
\end{eqnarray*}
We let
\[
p_1 := (-\Delta+V) \sum_{k \geq k_0} \eta_k \left(\int_0^{R_k}\left(1-\frac{t}{R_k}\right)\tilde{u}_k(t) dt\right).
\]
Then we have the decomposition
\[
\p_tu_L(0) = p_1 + \sum_{j=1}^{\frac{N-6}{2}} c_{j,k_1}^0 T_j^0 + \tilde{u}_1
\]
where
\[
-\tilde{u}_1 = f+\p_r g,
\]
where
\begin{equation}\label{fgdecomp}
f=\sum_{k\geq k_0} \eta_k f_k^1+f_k^2, \; g=\sum_{k \geq k_0} \tilde{g}_k.
\end{equation}
Note that by Hardy's inequality (see e.g. \cite{cdkm}),
\[
\norm{v}_{Z_{-2,R}} \lesssim \norm{\p_r v}_{Z_{-3,R}}.
\]
Therefore
\begin{eqnarray*}
\norm{\mathcal{A}( \p_t u_L(0))}_{Z_{-2,R}} & \leq & \norm{\mathcal{A} p_1}_{Z_{-2,R}} + \sum_{j=1}^{\frac{N-6}{2}}c_{j,k_1}^0 \norm{\mathcal{A} T_j^0}_{Z_{-2,R}} + \norm{\mathcal{A} \tilde{u}_1}_{Z_{-2,R}}.
\end{eqnarray*}
We will now bound each term: using \eqref{eq108}, \eqref{eq109}, \eqref{eq110}, \eqref{eq111},
\begin{eqnarray*}
\norm{\mathcal{A}(\eta_kf_k^1)}_{Z_{-2,R}} \lesssim \norm{\p_r\mathcal{A}(\eta_kf_k^1)}_{Z_{-3,R}} = \norm{(r\eta_kf_k^1)}_{Z_{-3,R}}
\end{eqnarray*}
and
\begin{eqnarray*}
\norm{\mathcal{A}(r\eta_kf_k^1)}_{L^2(R_k\leq r\leq 3R_k)}&=& \norm{r\eta_k\left(\sum_{j=0}^{\frac{N-6}{2}} (\bar{c}_{j,k}^{\infty}T_j^\infty+c_{j,k}^{\infty} \tilde{T}_j^0) - \frac{1}{R_k}\tilde{u}_k(R_k)\right)}_{L^2(R_k\leq r\leq 3R_k)}\\
&\lesssim& \left\langle \log \frac{R_k}{\langle R\rangle}\right\rangle \norm{u_L}_{\tilde{Y}_R}.
\end{eqnarray*}

We also have
\begin{eqnarray*}
\norm{\mathcal{A}(f_k^2)}_{Z_{-2,R}} \lesssim \norm{\p_r\mathcal{A}(f_k^2)}_{Z_{-3,R}} = \norm{(rf_k^2)}_{Z_{-3,R}}
\end{eqnarray*}
and
\[
\norm{rf_k^2}_{L^2(R_k\leq r\leq 3R_k)} \lesssim R_k^{-1} \int_0^{R_k} \norm{ \p_t u_L(0)}_{L^2(R_k \leq r \leq 3R_k)} dt \lesssim \norm{u_L}_{\tilde{Y}_R}.
\]
By definition of the $Z_{-2,R}$ norm,
\[
\norm{f}_{Z_{-2,R_{k_0}}} \lesssim \norm{u_L}_{\tilde{Y}_R}.
\]
On the other hand, since
\[
\mathcal{A}(\p_r \tilde{g}_k) = -r\tilde{g}_k(r) - \int_r^\infty \tilde{g}_k(r)dr,
\]
we have
\[
\norm{\mathcal{A}(\p_r\tilde{g}_k)}_{L^2(R_k \leq r\leq 3R_k)} \lesssim \int_0^{R_k} \norm{\p_tu_L(0)}_{L^2(R_k \leq r \leq 3R_k)} dt \lesssim R_k\norm{u_L}_{\tilde{Y}_R}.
\]
This implies $\norm{g}_{\tilde{Z}_{-2,R_{k_0}}} \lesssim \norm{u_L}_{\tilde{Y}_R}.$ Now to bound $p_1,$ note that
\begin{eqnarray*}
\mathcal{A}p_1 &=& \int_{r}^\infty \rho p_1(\rho) d\rho\\
&=& \sum_{k \geq k_0}\int_r^{\infty} \rho (-\Delta+V)\left(\eta_k \left(\int_0^{R_k}\left(1-\frac{t}{R_k}\right)\tilde{u}_k(t) dt\right)\right) d\rho\\
&=& \sum_{k \geq k_0} \int_r^\infty \rho V(\rho)\eta_k \left(\int_0^{R_k}\left(1-\frac{t}{R_k}\right)\tilde{u}_k(t)\right)\\
&-& r \p_r\left(\eta_k \int_0^{R_k}\left(1-\frac{t}{R_k}\right)\tilde{u}_k(t)\right) + (N-2)\eta_k\int_0^{R_k} \left(1-\frac{t}{R_k}\right)\tilde{u}_k(t)dt\\
&=:& \sum_{k\geq k_0}B_1^k+B_2^k+B_3^k.
\end{eqnarray*}
Since $V(\rho) \lesssim \rho^{-4}$ for $\rho>r,$ we have
\[
\norm{\sum_{k\geq k_0}B_1^k}_{Z_{-2},R},\norm{\sum_{k\geq k_0} B_3^k}_{Z_{-2},R} \lesssim \norm{u_L}_{\tilde{Y}_R}.
\]
To bound $B_2^k,$ note that
\[
\norm{B_2^k}_{L^2(R_k\leq r \leq 3R_k)} \lesssim R_k^{-1} \int_{0}^{R_k} \norm{\p_\rho (\eta_k \tilde{u}_k(t)}_{L^2(R_k\leq r \leq 3R_k)} \lesssim \norm{u_L}_{\tilde{Y}_R}
\] and so $\norm{\mathcal{A}(\p_t u_L(0))}_{Z_{-2}} \lesssim \norm{u_L}_{\tilde{Y}}.$
\end{proof}
Now we improve the bound with the following lemma.
\begin{lemma}\label{0improv}
Suppose that $u_L$ solves \eqref{LW}. We have
\[
\norm{\Pi_{L^2}^{\perp} \p_t u_L(0)}_{Z_{-4}} \lesssim \norm{\mathcal{A}(\Pi_{L^2}^{\perp} \p_t u_L(0))}_{Z_{-2}} + \norm{u_L}_{\tilde{Y}}
\]
\end{lemma}
\begin{proof}
First we claim we can suppose without loss of generality that $\p_t u_L(0) = \Pi_{L^2}^{\perp} \p_t u_L(0).$ Indeed, it suffices to note that if
\[
\Pi_{L^2}^\perp \p_t u_L(0) = \p_tu_L(0) - \sum_{k=0}^{\frac{N-6}{2}} c_k T_k^\infty,
\]
then
\[
\norm{\Pi_{L^2}^\perp \left(\p_tu_L(0) - \sum_{k=0}^{\frac{N-6}{2}} c_k T_k^\infty\right)}_{Z_{-3}} \lesssim \norm{\p_tu_L(0) - \sum_{k=0}^{\frac{N-6}{2}} c_k T_k^\infty}_{Z_{-3}}
\]
by direct computation, using that by Lemma \ref{varofpars}
\[
\vert T_k^\infty(r)\vert + \left|r\frac{d}{dr}T_k^\infty(r)\right| \lesssim \langle r\rangle^{-N+2+2k}.
\]
Now suppose then that $\p_tu_L(0) = \Pi_{L^2}^\perp \p_tu_L(0),$ and let $v=\mathcal{A}\p_t u.$ We also define the operator
\[
\Delta_N := \p_{rr} + \frac{N-1}{r} \p_r 
\]
Noting that $\Delta_N = \mathcal{A}^{-1}\Delta_{N-2}\mathcal{A}$ we see that $v$ solves
\[
\left\{
\begin{matrix}
\p_t^2 v-\Delta_{N-2} v + \mathcal{A}V\mathcal{A}^{-1} v = 0,\\
\overrightarrow{v}(0) = (\mathcal{A}\p_t u(0),0).
\end{matrix}
\right.
\]
Note that the kernel of the operator $-\Delta_{N-2} + \mathcal{A}V\mathcal{A}^{-1}$ is $\mathrm{Span}\{\mathcal{A}T_0^\infty,\int_1^r \rho T_0^0(\rho)d\rho,1\}.$
Let $k$ be a fixed integer, and let $R_k=2^k,$ $\tau_k=\frac{R_k}{2}.$ Let $\tilde{u}_k(t)$ be the solution to \eqref{LW} with initial condition $\Pi_{\dot{H}^1_{R_k}} u_L(t).$ By the proof of Lemma \ref{56equiv}, we can decompose the solution $u_L$ as
\begin{equation}\label{decomp3}
u_L(t) = \sum_{j=0}^{\frac{N-6}{2}} \alpha_{j,k}^\infty T_j^\infty + \alpha_{j,k}^0 \chi_0 T_j^0 + \tilde{u}_k(t)
\end{equation}
for $\vert t\vert \leq R_k.$ Also,
\[
\p_t v = -\mathcal{A}(-\Delta_N+V) u(t) = -\int_r^\infty - \rho \p_{\rho\rho} u -(N-1) \p_\rho u +Vu d\rho = r\p_ru(r) +(N-2)u(t)-\mathcal{A}Vu(t).
\]
Letting $\bar{T}_{j-1}^0 :=(-\Delta_N +V)(\chi_0 T_j^0),$ we see that by applying $\mathcal{A}(-\Delta_N+V)$ to \eqref{decomp3}
\[
\p_t v = \sum_{j=0}^{\frac{N-6}{2}} \alpha_{j+1,k}^\infty \mathcal{A} T_j^\infty + \alpha_{j+1,k}^0 \mathcal{A} \bar{T}_j^0 - \mathcal{A}(-\Delta_N +V)\tilde{u}_k(t).
\]
Let $q := \norm{\mathcal{A}( \p_t u(0))}_{Z_{-2,R}} + \norm{u}_{\tilde{Y}_R}.$ Recall from Lemma \ref{56equiv} $R_k := 2^k,$ $\tau_k=\frac{R_k}{2},$ and let $k_0$ such that $R_{k_0-1} \leq 3R \leq R_{k_0}.$ We have the following identity for any radial function $w$ and smooth test function $\psi:$
\begin{eqnarray*}
&& \int_{\R^N} \int_\R \psi(\p_{tt}w - \Delta w + Vw)w dtdx\\
&&= \int_{\R^N} \int_\R (\vert \nabla w\vert^2 -(\p_t w)^2)\psi dtdx + \int_{\R^N} \int_\R (V+\frac{1}{2}(\p_{tt}-\Delta))\psi w^2 dtdx.
\end{eqnarray*}
Let $\bar{\chi}$ be a cutoff such that $\bar{\chi}(r)=1$ for $r \in [3,4]$ and $\bar{\chi}(r)=0$ for $r \leq 2$ or $r \geq 10,$ and as before let $\bar{\chi}_k(r) = \bar{\chi} \left(\frac{r}{R_k}\right),$ $\tilde{\chi}_k(r) = \tilde{\chi} \left(\frac{r}{R_k}\right),$ $\psi_k(t,r)=\tilde{\chi}_k(t)\bar{\chi}_k(r),$ where $\tilde{\chi}(t)=1$ if $\vert t\vert \leq \frac{1}{2},$ and $\tilde{\chi}(t) = 0$ if $\vert t\vert \geq 1.$ Let $k \geq k_0.$ Note that $\mathrm{supp}(\bar{\chi}_k) \subseteq \{r \geq 6R\},$ $\mathrm{supp} \psi_k \subseteq \{r \geq \vert t\vert + R\}.$ Let
\[
\langle u,v\rangle_{L_{\bar{\chi}_k}^2} = \int_{\R^N} uv \bar{\chi}_k, \; \norm{u}_{L_{\bar{\chi}_k}^2} = \langle u,u\rangle_{L_{\bar{\chi}_k}^2}.
\]
Let $w_k = w_k(t),$ $b_{j,k} = b_{j,k}(t),\tilde{b}_{j,k} = \tilde{b}_{j,k}(t)$ be uniquely defined so that
\[
u_L(t) = w_k(t) + \sum_{j=0}^{\frac{N-6}{2}} b_{j,k}T_j^\infty + \tilde{b}_{j,k} \tilde{Y}_j^0
\]
where $\tilde{Y}_j^0 = \chi_0 T_j^0,$
and
\[
w_k\perp_{L^2_{\chi_k}} \{T_j^\infty,\tilde{Y}_{j}^0\}_{j=0}^{\frac{N-6}{2}} \mbox{ for }k\leq -1,\; w_k\perp_{L^2_{\chi_k}} \{T_j^\infty\}_{j=0}^{\frac{N-6}{2}} \mbox{ for }k\geq 0.
\]
Then note that
\begin{eqnarray*}
\int_{\R^N} \int_\R \vert \nabla w_k\vert^2 \psi_k dtdx &=& \int_{\R^N} \int_\R (\p_t w_k)^2 \psi_k dtdx - \int_{\R^N} \int_\R \left(V+\frac{1}{2}(\p_{tt}-\Delta)\right)\psi_k w_k^2 dtdx\\
&+& \int_{\R^N} \int_\R \psi_k (\p_{tt}w_k - \Delta w_k + Vw_k) w_k dtdx.
\end{eqnarray*}
We have
\begin{eqnarray*}
&&\int_{\R^N} \int_\R \psi_k (\p_{tt}w_k - \Delta w_k +V w_k)w_k dtdx\\
&& -\int_{\R^N} \int_\R \psi_k\left(-\sum_{j=0}^{\frac{N-6}{2}} b_{j,k}(-\Delta + V)T_j^\infty + \tilde{b}_{j,k} (-\Delta + V)\tilde{Y}_j^0 - \sum_{j=0}^{\frac{N-6}{2}} (\p_{tt}b_{j,k}) T_j^\infty - (\p_{tt}\tilde{b}_{
j,k})\tilde{Y}_j^0\right) w_kdtdx\\
&&=0
\end{eqnarray*}
since $w_k \perp_{L_{\chi_k}^2} T_j^\infty,\tilde{Y}_j^0.$ Therefore
\begin{equation}\label{energyid10}
\int_{\R^N} \int_\R \vert \nabla w_k\vert^2 \psi_k dtdx = \int_{\R^N} \int_\R (\p_t w_k)^2 dtdx-\int_{\R^N} \int_\R (V+\frac{1}{2}(\p_{tt}-\Delta))\psi_k w_k^2dtdx.
\end{equation}
Let
\[
g_k = \p_t v - \p_t \left(\sum_{j=0}^{\frac{N-6}{2}} b_{j,k} T_j^\infty + \tilde{b}_{j,k} \tilde{Y}_j^0\right).
\]
Now, there exist coefficients $b_{j,k}, \tilde{b}_{j,k}, \beta_{j,k}^\infty, \beta_{j,k}^0$ such that
\begin{equation}\label{eq104}
v(0) = w_k(0)+\sum_{j=0}^{\frac{N-6}{2}} b_{j,k}(0)T_j^\infty + \tilde{b}_{j,k}(0) \tilde{Y}_j^0
\end{equation}
and
\begin{equation}\label{eq105}
g_k = \p_t w_k(t) + \sum_{j=0}^{\frac{N-6}{2}} \beta_{j,k}^\infty(0) T_j^\infty + \beta_{j,k}^0(0) T_j^0,
\end{equation}
where $\beta_{j,k}$ ar ethe unique coefficients such that the above identity holds, and where
\[
w_k(0) = \Pi_{L^2_{R_k}} v(0).
\]
Using the definition of $Z_{-2,R},$ we obtain
\begin{equation}\label{eq102}
\norm{v_0}_{L_{\bar{\chi}_k}^2}\lesssim \norm{v_0}_{L^2(2R_k \leq r \leq 10R_k)} \lesssim R_k \langle \log \frac{R_k}{\langle R\rangle}\rangle\norm{v_0}_{Z_{-2,R}} \leq R_k \langle \log \frac{R_k}{\langle R\rangle}\rangle q
\end{equation}
and similarly
\begin{equation}\label{eq103}
\norm{g_k}_{L_{\bar{\chi}_k}^2}\lesssim \norm{g_k}_{L^2(2R_k \leq r \leq 10R_k)} \lesssim \norm{g_k}_{L_{2R_k}^2} \leq \norm{\p_t v}_{Y_R} \leq q .
\end{equation}
Now by Pythagoras,
\[
\norm{ \sum_{j=0}^{\frac{N-6}{2}}b_{j,k} T_j^\infty + \tilde{b}_{j,k} \tilde{Y}_j^0}_{L_{\bar{\chi}_k}^2} \leq \norm{u_0}_{L_{\bar{\chi}_k}^2} \lesssim R_k \left\langle \log \frac{R_k}{\langle R\rangle}\right\rangle q,
\]
\[
 \norm{ \sum_{j=0}^{\frac{N-6}{2}}\beta_{j,k}^\infty T_j^\infty + \beta_{j,k}^0 \tilde{Y}_j^0}_{L_{\bar{\chi}_k}^2}\leq \norm{g_k(t)}_{L_{\bar{\chi}_k}^2} \lesssim q.
\]
Note that for general $\gamma_{j,k}^\infty, \gamma_{j,k}^0,$ we have
\[
\norm{ \sum_{j=0}^{\frac{N-6}{2}}\gamma_{j,k}^\infty T_j^\infty + \gamma_{j,k}^0 \tilde{Y}_j^0}_{L^2(2R_k \leq r \leq 10R_k)} \lesssim \norm{ \sum_{j=0}^{\frac{N-6}{2}}\gamma_{j,k}^\infty T_j^\infty +\gamma_{j,k}^0 \tilde{Y}_j^0}_{L_{\bar{\chi}_k}^2}.
\]
This implies in particular that
\begin{equation}\label{eq100}
\norm{ \sum_{j=0}^{\frac{N-6}{2}}b_{j,k}^\infty T_j^\infty + b_{j,k}^0 \tilde{Y}_j^0}_{L^2(2R_k \leq r \leq 10R_k)} \lesssim R_k \left\langle \log \frac{R_k}{\langle R\rangle}\right\rangle q,
\end{equation}
\begin{equation}\label{eq101}
\norm{ \sum_{j=0}^{\frac{N-6}{2}}\beta_{j,k}^\infty T_j^\infty + \beta_{j,k}^0 \tilde{Y}_j^0}_{L^2(2R_k \leq r\leq 10R_k)} \lesssim q.
\end{equation}
Putting \eqref{eq100}, \eqref{eq101}, \eqref{eq102}, \eqref{eq103} together with $\eqref{eq104}, \eqref{eq105}$ we obtain
\[
\norm{w_k(t)}_{L^2(2R_k \leq r \leq 10 R_k)} \lesssim R_k \left\langle \log \frac{R_k}{\langle R\rangle}\right\rangle q,
\]
\[
\norm{\p_t w_k(t)}_{L^2(2R_k \leq r \leq 10 R_k)} \lesssim q.
\]
Combining this with \eqref{energyid10} we obtain
\[
\int_{\R^N} \int_\R \vert \nabla w_k\vert^2 \psi_k dtdx \lesssim R_k \left\langle \log \frac{R_k}{\langle R\rangle}\right\rangle^2 q^2.
\]
By the mean value theorem, there is a $t_k$ with $\vert t_k \vert \leq \frac{R_k}{2}$ such that
\[
\int_{3R_k \leq r\leq 9 R_k} \vert \nabla w(t_k) \vert^2 dx \lesssim \left\langle \log \frac{R_k}{\langle R\rangle}\right\rangle^2 q^2.
\]
This bounds $\nabla w_k.$ To bound $\nabla v(0),$ by finite propagation speed we have for $\vert t\vert \leq t_k,$ $4R_k \leq r \leq 8 R_k,$ that
\[
v(t_k) = w_k(t_k)+ \sum_{j=0}^{\frac{N-6}{2}} h_{j,k} T_j^\infty + \tilde{h}_{j,k} T_j^0,
\]
\[
\p_t v(t_k) = g_k(t_k)+ \sum_{j=0}^{\frac{N-6}{2}} a_{j,k} T_j^\infty + \tilde{a}_{j,k} \tilde{Y}_j^0
\]
for appropriate coefficients $h_{j,k} , \tilde{h}_{j,k}, a_{j,k},\tilde{a}_{j,k}.$ Hence
\[
v(t) = \bar{w}_k(t) + \sum_{j=0}^{\frac{N-6}{2}} (h_{j,k} + (t-t_{j,k})a_{j,k})T_j^\infty + (\tilde{h}_{j,k} + (t-t_{j,k})\tilde{a}_{j,k})\tilde{Y}_j^0
\]
where $\bar{w}_k$ solves
\[
\left\{
\begin{matrix}
\p_{tt}\bar{w}_k - \Delta \bar{w}_k + V \bar{w}_k = 0,\\
(\bar{w}_k,\p_t \bar{w}_k)\vert_{t=t_k} = (w_k(t_k),g_k(t_k)).
\end{matrix}
\right.
\]
Letting
\[
d_{j,k} = h_{j,k}(t_k)-t_k a_{j,k}(t_k), \; \tilde{d}_{j,k} = \tilde{h}_{j,k}(t_k)-t_k \tilde{a}_{j,k}(t_k)
\]
we see that for all $4R_k \leq r\leq 8 R_k,$
\[
v_0 = \bar{w}_k(0)+\sum_{j=0}^{\frac{N-6}{2}} d_{j,k} T_j^\infty + \tilde{d}_{j,k} \tilde{Y}_j^0.
\]
Let $(\tilde{w}_{k,0},\tilde{w}_{k,1})$ be such that
\[
(\tilde{w}_{k,0},\tilde{w}_{k,1})(r) = (w_k,g_k)(t,r)
\]
for all $r \in [3R_k , 9R_k],$ and
\[
\norm{(\tilde{w}_{k,0},\tilde{w}_{k,1})}_{\dot{H}^1\times L^2} \lesssim \int_{3R_k\leq \vert x\vert \leq 9 R_k} \left(\vert \nabla v_k(t_k)\vert^2 + \frac{\vert v_k(t_k)\vert^2}{R_k^2} + \vert g_k(t_k)\vert^2\right) dx.
\]
Also let $\tilde{w}_k$ be the solution to
\[
\left\{
\begin{matrix}
\p_t^2 \tilde{w}_k -\Delta \tilde{w}_k + V\tilde{w}_k = 0,\\
(\tilde{w}_k(t_k),\p_t \tilde{w}_k(t_k)) = (\tilde{w}_{k,0},\tilde{w}_{k,1}).
\end{matrix}
\right.
\]
Then again by finite propagation speed $\bar{w}_k(0,r) = \tilde{w}_k(0,r)$ for all $4R_k \leq r \leq 8 R_k.$ After replacing $\tilde{w}_k$ by $\tilde{w}_k$ multiplied by a cutoff at $r=R_k+|t|$ and standard energy estimates, we obtain
\[
\norm{\tilde{w}_k(0)}_{\dot{H}^1_{R_k}} \lesssim \norm{(\tilde{w}_{k,0},\tilde{w}_{k,1})}_{\mathcal{H}}.
\]
Now applying the previous estimates, we obtain
\[
\norm{\tilde{w}_k(0)}_{\dot{H}^1_{R_k}} \lesssim \left\langle \log \frac{R_k}{\langle R\rangle}\right\rangle q
\]
and thus by Hardy, for all $4R_k \leq r \leq 8 R_k$ we have
\[
\norm{\nabla \bar{w}_k(0)}_{L^2(4R_k \leq r\leq 8 R_k)} + R_k^{-1} \norm{\bar{w}_k(0)}_{L^2(4R_k \leq r\leq 8 R_k)} \lesssim \left\langle \log \frac{R_k}{\langle R\rangle}\right\rangle q.
\]
By definition of the $Z_{-2}$ norm, we have
\[
\norm{\sum_{j=0}^{\frac{N-6}{2}} d_{j,k} T_j^\infty + \tilde{d}_{j,k} \tilde{Y}_j^0}_{L^2(4R_k \leq r\leq 8 R_k)} \lesssim R_k \left\langle \log \frac{R_k}{\langle R\rangle}\right\rangle q.
\]
Therefore
\[
\sum_{j=0}^{\frac{N-6}{2}} \norm{d_{j,k} T_j^\infty}_{L^2(4R_k \leq r\leq 8 R_k)} + \norm{\tilde{d}_{j,k} \tilde{Y}_j^0}_{L^2(4R_k \leq r\leq 8 R_k)} \lesssim R_k \left\langle \log \frac{R_k}{\langle R\rangle}\right\rangle q
\]
and
\[
\sum_{j=0}^{\frac{N-6}{2}} \norm{d_{j,k} T_j^\infty}_{\dot{H}^1(4R_k \leq r\leq 8 R_k)} + \norm{\tilde{d}_{j,k} \tilde{Y}_j^0}_{\dot{H}^1(4R_k \leq r\leq 8 R_k)} \lesssim q.
\]
These imply
\[
\norm{\nabla v_0}_{L^2(4R_k \leq r \leq 8R_k)} \lesssim \left\langle \log \frac{R_k}{\langle R\rangle}\right\rangle q
\]
which implies the result, since $4R_{k_0} \leq 24R.$
\end{proof}
Now Lemma \ref{otherpart} follows from Lemma \ref{0improv}, \ref{56equiv} and Lemma \ref{notinterestingsingle}.

\subsection{The $N\equiv 2 \mod 4$ case}\label{other2}
This part is similar to the $N \equiv 0 $ case. In this case we need an additional lemma about an elliptic estimate:
\begin{lemma}
There exists a constant $C>0$ such that the following holds. Let $R>0,$ and suppose $u \in Z_{-2,R}(\R^N)$ solves for $r>R$
\[
-\Delta u + Vu = f+\p_r g
\]
where $u,f,g$ are radial and satisfy
\[
\norm{f}_{Z_{-4,R}}, \norm{g}_{Z_{-3,R}} < \infty.
\]
Then if $R \geq 1,$ then
\[
d_{Z_{-2,R}}(u,\mathrm{span}\{T_0^\infty , \ldots , T_{\frac{N-6}{2}}^\infty\}) \leq C (\norm{f}_{Z_{-4,R}}+\norm{g}_{Z_{-3,R}}),
\]
while if $0 < R<1,$ we have
\[
d_{Z_{-2,R}}(u,\mathrm{span}\{T_0^\infty , \ldots , T_{\frac{N-6}{2}}^\infty, T_0^0,\ldots,T_{\frac{N-6}{2}}^0\}) \leq C (\norm{f}_{Z_{-4,R}}+\norm{g}_{Z_{-3,R}}).
\]
\end{lemma}
The quantities $\norm{\Pi_{L^2}^{\perp} \p_t u(0)}_{Z_{-4}}$ are replaced by $d_{Z_{-2,R}}(u,\mathrm{span}\{T_0^\infty , \ldots , T_{\frac{N-6}{2}}^\infty\}).$ The proof of Lemma \ref{0improv} becomes slightly easier, since the averaging operator $\mathcal{A}$ is not needed to reduce to an $N-2$ dimensional radial wave equation. The proof follows  from the same argument used in the case $N \equiv 0\mod 4.$

\section{The multisoliton case}

In this section we prove Theorem \ref{mainmulti}.

First, we note that the problem  \eqref{multieq} is globally well posed. Indeed, local well posedness follows from the Strichartz estimate Lemma \ref{extstrichartz}., and global well posedness follows from the fact that \eqref{multieq} is linear. Now we have the following lemma.
\begin{lemma}\label{multistrich}
For all $J\in \N,$ there is a $C>0$ such that for all $\lambda \in \Lambda_J,$ if $u$ is a solution in $\R^{1+N}$ of
\[
\left\{
\begin{matrix}
\p_{tt} u-\Delta u + V_\lambda u = f,\\
(u,\p_t u)\vert_{t=0} = (u_0,u_1) \in \dot{H}^1\times L^2
\end{matrix}
\right.
\]
where $f \in L_t^1 L_x^2(\R\times \R^N),$ then
\[
\sup_{t \in \R} \norm{(u,\p_t u)}_{\mathcal{H}_{R+\vert t\vert}} + \norm{u}_{L_t^{\frac{2(N+1)}{N-2}} L_x^{\frac{2(N+1)}{N-2}}(r\geq R+\vert t\vert)} \lesssim \norm{(u_0,u_1)}_{\mathcal{H}_R} + \norm{f}_{L_t^1 L_x^2(r \geq R+\vert t\vert)}.
\]
\end{lemma}
\begin{proof}
The proof is as in the proof of Lemma 2.8 in \cite{solres}.
\end{proof}
Next we claim the following inequalities.
\begin{claim}
Suppose that $f \in \dot{H}^1$ is radial. Then $\nabla f \in Z_{-3,\lambda}$ with
\begin{equation}
\norm{\nabla f}_{Z_{-3,\lambda}} \leq \norm{f}_{\dot{H}^1}
\end{equation}
and
\begin{equation}
\vert f(r)\vert \lesssim \frac{1}{r^2} \left(1+ \left\vert \min_{1\leq j \leq J} \log \left(\frac{r}{\lambda_j}\right)\right\vert\right) \norm{\nabla f}_{Z_{-3,\lambda}}.
\end{equation}
Moreover, if $\chi^0,\chi^1$ are two smooth cutoffs with $\chi^0(r)=1$ if $r \leq 1,$ $\chi^0(r)=0$ for $r \geq 2,$ $\chi^1$ with compact support outside the origin, and $\chi_R^i(r):= \chi^i\left(\frac{r}{R}\right),$ then
\begin{equation}\label{H1cutoff}
\norm{\chi_R^1 f}_{\dot{H}^1} \leq C \left(1+ \left\vert \min_{1\leq j \leq J} \log \left( \frac{R}{\lambda_j}\right)\right\vert\right) \norm{\nabla f}_{Z_{-3,\lambda}}
\end{equation}
and
\begin{equation}
\norm{\nabla(\chi_R^0 f)}_{Z_{-3,\lambda}} \leq C \norm{\nabla f}_{Z_{-3,\lambda}},
\end{equation}
\begin{equation}
\norm{\nabla((1-\chi_R^0) f)}_{Z_{-3,\lambda}} \leq C \norm{1_{\vert x\vert \geq R}\nabla f}_{Z_{-3,\lambda}},
\end{equation}
where $C$ is independent of $R.$
\end{claim}
The proof of this claim is the same as that of Claim 4.2 in \cite{cdkm}.
For the moment we will consider a single soliton. We have the following lemma.
\begin{lemma}\label{L43}
There exists $C>0$ such that the following holds. For any $\lambda>0$ and any $u_0\in \dot{H}^1,$ the solution $u$ to
\[
\left\{
\begin{matrix}
\p_t^2 u -\Delta u + V_{\{\lambda\}} u =0\\
(u(0),\p_tu(0)) = (u_0,0)
\end{matrix}
\right.
\]
satisfies
\begin{equation}
\norm{V_{\{\lambda\}} u}_{L^1L^2(r>\vert t\vert)} \leq C\norm{\nabla u_0}_{Z_{-3}}.
\end{equation}
If $u_0(r)=0$ for all $r \geq R$ for some $R \leq 1,$ then
\[
\norm{V_{\{\lambda\}}u}_{L^1L^2(r \geq \vert t\vert)} \leq CR^{\frac{3}{4N}}\norm{\nabla u_0}_{Z_{-3}}.
\]
If $u_0(0)$ for all $r \leq R$ for some $R \geq 1,$ then
\[
\norm{V_{\{\lambda\}}u}_{L^1L^2(r > \vert t\vert)} \leq CR^{-1} \norm{\nabla u_0}_{Z_{-3}}.
\]
Similarly, if $u_1\in L^2$ and $u$ is the solution to
\[
\left\{
\begin{matrix}
\p_{tt} u-\Delta u + V_\lambda u = f,\\
(u,\p_t u)\vert_{t=0} = (0,u_1) \in \dot{H}^1\times L^2
\end{matrix}
\right.
\]
satisfies
\begin{equation}
\norm{V_{\{\lambda\}} u}_{L^1L^2(r>\vert t\vert)} \leq C\norm{u_1}_{Z_{-3}}.
\end{equation}
If $u_1(r)=0$ for all $r \geq R$ for some $R \leq 1,$ then
\[
\norm{V_{\{\lambda\}}u}_{L^1L^2(r \geq \vert t\vert)} \leq CR^{\frac{3}{4N}}\norm{u_1}_{Z_{-3}}.
\]
If $u_1(0)=0$ for all $r \leq R$ for some $R \geq 1,$ then
\[
\norm{V_{\{\lambda\}}u}_{L^1L^2(r > \vert t\vert)} \leq CR^{-1} \norm{u_1}_{Z_{-3}}.
\]
\end{lemma}
\begin{proof}
As in the proof of Lemma 4.3 in \cite{cdkm}, we use a dyadic partition. Let $\chi$ be such that $\mathrm{Supp} \chi \subset \left[\frac{1}{2},2\right]$ and $1 = \sum_{\ell\in \Z} \chi_{2^\ell}(r).$ We let $u_L$ be the solution to
\[
\left\{
\begin{matrix}
\p_t^2 u_\ell - \Delta u_{\ell} + V_{(\ell)} u_\ell = 0\\
(u_\ell(0),\p_tu_\ell(0)) = (u_{\ell,0},0)
\end{matrix}
\right.
\]
where $u_{\ell,0} = \chi_{2^\ell} u_0.$ Note that $u = \sum_{\ell \in\Z} u_\ell.$ Now by \eqref{H1cutoff}, we have
\[
\norm{u_{\ell,0}}_{\dot{H}^1} \lesssim (1+\vert \ell\vert) \norm{\nabla u_0}_{Z_{-3}}.
\]
Since $\mathrm{Supp} u_{0,\ell} \subset \left[ 2^{\ell-1},2^{\ell+1}\right]$ we have that
\[
\mathrm{Supp} (u_{0,\ell}) \subseteq \{\max(0, 2^{\ell-1}-\vert t\vert) \leq r \leq 2^{\ell+1} + \vert t\vert\}.
\]
Now by H\"{o}lder and Lemma \ref{multistrich} we obtain
\[
\norm{V_{\{\lambda\}}u_L}_{L^1L^2(r \geq \vert t\vert)} \lesssim (1+\vert \ell\vert) \norm{V}_{L_t^2 L_x^{\frac{2N}{3}}(S_\ell)} \norm{\nabla u_0}_{Z_{-3}}
\]
where
\[
S_\ell := \{\max(\vert t\vert , 2^{\ell-1}-\vert t\vert) \leq r\leq 2^{\ell+1} + \vert t\vert\}.
\]
If $\ell \leq 0,$ then
\begin{eqnarray*}
\norm{V}_{L_t^2L_x^{\frac{2N}{3}}(S_\ell)} \leq \norm{V}_{L_t^2 L_x^{\frac{2N}{3}} (\vert t\vert \leq r\leq \vert t\vert + 2^{\ell+1})} \lesssim 2^{\frac{3}{2N}\ell}.
\end{eqnarray*}
Similarly, if $\ell \geq 0,$ we have
\begin{eqnarray*}
\norm{V}_{L_t^2L_x^{\frac{2N}{3}}(S_\ell)} \leq \norm{V}_{L_t^2 L_x^{\frac{2N}{3}} (\max(\vert t\vert , 2^{\ell-2}) \leq r\leq \vert t\vert + 2^{\ell+1})} \lesssim 2^{-2\ell}.
\end{eqnarray*}
This implies that
\begin{eqnarray*}
\norm{V_{\{\lambda\}} u}_{L^1L^2(r \geq \vert t\vert)} & \lesssim& \sum_{\ell \in \Z} \norm{V_{\{\lambda\}} u_\ell}_{L^1L^2(r \geq \vert t\vert)} \\
&\leq & \norm{\nabla u_0}_{Z_{-3}}\left(\sum_{\ell \leq 0} (1+\vert \ell\vert) 2^{\frac{3\ell}{2N}} + \sum_{\ell \geq 0} 2^{-2\ell}\right) \lesssim \norm{\nabla u_0}_{Z_{-3}}.
\end{eqnarray*}
Now if we additionally assume that $u_0(r) = 0$ if $r \geq R$ for some $R \geq 1,$ we obtain
\[
\norm{V_{\{\lambda\}} u}_{L^1L^2(r \geq \vert t\vert)}\lesssim \norm{\nabla u_0}_{Z_{-3}} \sum_{\ell \leq \ell_0} (1+\vert \ell\vert)2^{\frac{3\ell}{2N}} \lesssim 2^{\frac{3\ell_0}{4N}} \norm{\nabla u_0}_{Z_{-3}} \lesssim R^{\frac{3}{4N}}\norm{\nabla u_0}_{Z_{-3}}.
\]
The last assertion follows similarly.
\end{proof}

We have the following lemma.
\begin{lemma}\label{distlemma}
Let $J \in \N.$ There exists a small enough $\gamma^* >0$ such that the following holds for any $\lambda \in \Lambda_J$ with $\gamma(\lambda) \leq  \gamma^*.$ For $1\leq j \leq J,$ let $\overrightarrow{\phi}_{j}$ and $\overrightarrow{\psi}_{j}$ denote the solutions to \eqref{multieq} with initial data $(0,(\Lambda W)_{[\lambda_j]})$ and $((\Lambda W)_{(\lambda_j)},0)$ respectively. Then for any $1\leq j \leq J,$ we have
\[
\sup_{t \in \R} \norm{\overrightarrow{\phi}_{j,k} - (t(\Lambda W)_{[\lambda_j]},\Lambda W_{[\lambda_j]})}_{\mathcal{H}_{\vert t\vert}} \lesssim \gamma(\lambda)
\]
\[
\sup_{t \in \R} \norm{\overrightarrow{\psi}_{j,k} - (\Lambda W_{(\lambda_j)},0)}_{\mathcal{H}_{\vert t\vert}} \lesssim \gamma(\lambda).
\]
\end{lemma}
The proof follows from direct computations, which we omit here.

\subsection{Non resonant component}
In this section we will consider solutions satisfying $u(t,r) = (-1)^{\frac{N}{2}}u(-t,r).$ We will show the following lemma.

\begin{prop}\label{multifirsthalf}
For any $J\in\N,$ there are $\gamma^*,C>0$ such that for all $\lambda \in \Lambda_J$ with $\gamma(\lambda) \leq \gamma^*$
any radial solution to \eqref{multieq} satisfies
\[
\norm{\Pi_{\dot{H}^1,\lambda}^\perp u_0}_{\dot{H}^1}^2 \lesssim \sum_{\pm} \lim_{t \to \pm \infty} \int_{r \geq \vert t\vert} \vert \nabla_{t,x} u\vert^2 + \gamma(\lambda)\norm{u_0}_{\dot{H}^1}^2
\]
in the case $N \equiv 0 \mod 4,$ and
\[
\norm{\Pi_{L^2,\lambda}^\perp u_0}_{L^2}^2 \lesssim \sum_{\pm} \lim_{t \to \pm \infty} \int_{r \geq \vert t\vert} \vert \nabla_{t,x} u\vert^2 + \gamma(\lambda) \norm{u_0}_{L^2}^2
\]
if $N \equiv 2 \mod 4.$
\end{prop}
For the proof we will need the following lemma.
For $R>0, \mu>0,$ let $\bar{P}_{\mu,R}^2$ be the orthogonal projection on the orthogonal of $\mathrm{Span}\{(T_k^\infty)_{(\mu)}:0\leq k \leq \frac{N-6}{2}\}$ in $\dot{H}^1_R$ if $N \equiv 0 \mod 4$ and the orthogonal projection on the orthogonal of $\mathrm{Span}\{(T_k^\infty)_{[\mu]}:0\leq k \leq \frac{N-6}{2}\}$ in $L^2_R$ if $N \equiv 2 \mod 4.$
\begin{lemma}\label{claim46}
For all $\eta\geq 0,$ there is a $C_\eta>0$ such that for all $\mu>0$ and a solution $u$ of
\[
\p_t^2u-\Delta u + V_{\{\mu\}} u =f, \; \vert x\vert > \eta\mu+\vert t\vert
\]
with $\overrightarrow{u}(0)\in \mathcal{H}, f \in L^1_tL^2_x,$ we have
\[
\norm{\bar{P}_{\mu,\eta\mu} \p_t u(0)}_{\dot{H}_{\eta\mu}^1}^2 \leq C_\eta \norm{1_{\vert x\vert > \eta \mu+\vert t\vert}f}_{L^1_tL^2_x}^2 + C_\eta\sum_{\pm} \lim_{t\to \pm \infty} \int_{\vert x\vert > \eta \mu +\vert t\vert} \vert \nabla_{t,x}u(t,x)\vert^2 dx
\]
if $N \equiv 0 \mod 4,$ and
\[
\norm{\bar{P}_{\mu,\eta\mu} \p_t u(0)}_{L_{\eta\mu}^2}^2 \leq C_\eta \norm{1_{\vert x\vert > \eta \mu+\vert t\vert}f}_{L^1_tL^2_x}^2 + C_\eta\sum_{\pm} \lim_{t\to \pm \infty} \int_{\vert x\vert > \eta \mu +\vert t\vert} \vert \nabla_{t,x}u(t,x)\vert^2 dx
\]
\end{lemma}
\begin{proof}
By scaling we can assume $\mu = 1.$ By time symmetry, we can assume $\p_t u(0) = 0$ if $N \equiv 0 \mod 4$ and $u(0)=0$ if $N \equiv 2 \mod 4.$ Now let $u_{\mathrm{hom}}$ be the solution to
\[
\left\{
\begin{matrix}
\p_t^2 u_{\mathrm{hom}} -\Delta u_{\mathrm{hom}} + V_{\{\mu\}} u_{\mathrm{hom}}=0\\
(u_{\mathrm{hom}}(0),\p_t (u_{\mathrm{hom}}(0))) = (u(0),\p_tu(0))
\end{matrix}
\right.
\]
and $\bar{u}$ the solution to
\[
\left\{
\begin{matrix}
\p_t^2 \bar{u} - \Delta \bar{u} + V_{\{\mu\}}\bar{u} = f,\\
(\bar{u}(0),\p_t \bar{u}(0))=(0,0).
\end{matrix}
\right.
\]
Then if $N \equiv 0 \mod 4,$
\[
\norm{\Pi_{\mu,\eta\mu}^\perp u(0)}_{\dot{H}^1_{\eta\mu}}^2 \lesssim \norm{\Pi_{\mu,\eta\mu}^\perp u_{\mathrm{hom}}}_{\dot{H}^1_{\eta\mu}}^2 + \norm{\bar{u}(0)}_{\dot{H}_{\eta\mu}^1}^2
\]
and
\[
\norm{\Pi_{\mu,\eta\mu}^\perp u(0)}_{L^2_{\eta\mu}(0)}^2 \lesssim \norm{\Pi_{\mu,\eta\mu}^\perp u_{\mathrm{hom}}}_{L^2_{\eta\mu}}^2 + \norm{\bar{u}(0)}_{L^2_{\eta\mu}}^2
\]
if $N \equiv 2 \mod 4.$
Now let $\tilde{u}$ be the solution to
\[
\left\{
\begin{matrix}
\p_t^2\tilde{u} - \Delta \tilde{u} + V_{\{\mu\}}\tilde{u} = f1_{\vert x\vert > \eta\mu+\vert t\vert},\\
(\tilde{u}(0),\p_t\tilde{u}(0))=(0,0).
\end{matrix}
\right.
\]
By finite propagation speed,
\[
\norm{\bar{u}}_{\dot{H}_{\eta\mu}^1} + \norm{\p_t \bar{u}}_{L_{\eta\mu}^2}^2 \lesssim \norm{\tilde{u}}_{\dot{H}_{\eta\mu}^1} + \norm{\p_t \tilde{u}}_{L_{\eta\mu}^2}^2.
\]
Now by standard energy estimates,
\[
\norm{\tilde{u}}_{\dot{H}^1}^2 + \norm{\p_t \tilde{u}}_{L^2}^2 \lesssim \norm{f}_{L^1L^2(\vert x\vert > \eta\mu+\vert t\vert)}.
\]
Therefore it suffices to show the statement for $u_{\mathrm{hom}},$ that is for the case $f \equiv 0.$

Additionally note that if $\eta= 0,$ then the result is Lemma \ref{notinterestingsingle} for $R=0.$ We will then assume $\eta>0.$ Using Lemma \ref{notinterestingsingle} we can decompose $u_0=u(0)$ as
\[
u_0 = \sum_{k=0}^{\frac{N-6}{2}} c_k^\infty T_k^\infty + c_k^0 T_k^0 + \Pi^\perp u,
\]
with
\[
\norm{\Pi^\perp u}_{\mathcal{A}_N}^2 \lesssim E_{\mathrm{out}}(\eta):= \lim_{t\to \infty} \int_{\vert x\vert > \eta+\vert t\vert} \vert \nabla_{t,x}u(t,x)\vert^2 dx.
\]
Rewriting this as
\[
\sum_{k=0}^{\frac{N-6}{2}}c^0_k T_k^0 =u_0-\sum_{k=0}^{\frac{N-6}{2}} c_k^\infty T_k^\infty - \Pi^\perp u
\]
and taking the exterior energy as in the proof of Lemma \ref{notinterestingsingle} together with \eqref{eq12} we obtain that
\[
c^0_k\lesssim (E_{\mathrm{out}}(\eta))^{\frac{1}{2}}.
\]
This implies that
\[
\norm{\Pi^\perp u}_{\mathcal{A}_N}^2+\sum_{k=0}^{\frac{N-6}{2}}(c_k^0)^2 \lesssim E_{\mathrm{out}}(\eta).
\]
This completes the proof.
\end{proof}
Now we will prove Proposition \ref{multifirsthalf}.
\begin{proof}[Proof of Proposition \ref{multifirsthalf}.]
We can proceed as in the proof of Lemma 4.5 in \cite{cdkm}. First we can assume by time symmetry that
\[
(u,\p_t u)\vert_{t=0} = \left\{
\begin{matrix}
(u_0,0) \mbox{ if } N\equiv 0, \mod 4\\
(0,u_1) \mbox{ if } N \equiv 2 \mod 4.
\end{matrix}
\right.
\]
Moreover, we can decompose
\[
u_0 = \Pi_{\mu}^\perp u + v_0
\]
with
\[
v_0 \in \mathrm{Span}\left\{(T_k^\infty)_{(\lambda_j)}:0 \leq k \leq \frac{N-6}{2}, 1\leq j \leq J\right\}
\]
if $N \equiv 0 \mod 4,$
\[
u_1 = \Pi_{\mu}^\perp u + v_1
\]
with
\[
v_1 \in \mathrm{Span}\left\{(T_k^\infty)_{[\lambda_j]}:0 \leq k \leq \frac{N-6}{2}, 1\leq j \leq J\right\}
\]
if $N \equiv 2 \mod 4.$ Let $v$ be the solution to \eqref{LW} with initial condition $(v_0,0)$ if $N \equiv 0 \mod 4$ and $(0,v_1)$ if $N \equiv 2 \mod 4,$ we see by Lemma \ref{distlemma} that
\[
\lim_{t \to \infty} \int_{\vert x\vert > \vert t\vert} \vert \nabla_{t,x} v(t,x)\vert^2 dx \lesssim \gamma(\lambda) \norm{(v_0,v_1)}_{\dot{H}^1\times L^2}^2 \leq \gamma(\lambda) \norm{(u_0,u_1)}_{\dot{H}^1\times L^2}^2.
\]
Thus we can reduce to the case that $v=0.$

We have
\[
\p_t^2 u -\Delta u - V_{\{\lambda_j\}} u = \sum_{k \neq j} V_{\{\lambda_k\}} u.
\]
By Lemma \ref{claim46}, we have for $N \equiv 0 \mod 4$
\[
\norm{P_{\lambda_j,\eta_j\lambda_j}u_0}_{\dot{H}^1_{\eta_j\lambda_j}} \leq C_{\eta_j} \sum_{k \neq j} \norm{1_{\vert x\vert > \eta_j \lambda_j+\vert t\vert}V_{\{\lambda_k\}}u}_{L^1L^2} + C_{\eta_j}\sqrt{E_{\mathrm{out}}},
\]
where $\bar{P}_{\lambda,\rho}^\perp$ is the $\dot{H}^1_\rho$ (respectively $L_\rho^2$) projection onto the orthogonal of $\mathrm{Span}\{(T_k^\infty)_{(\lambda_j)}\}_{k=0}^{\frac{N-6}{2}}$ for $N \equiv 0 \mod 4$ (respectively $N \equiv 2 \mod 4$), with a similar statement for $N \equiv 2 \mod 4.$ If $k> j,$ we have
\begin{eqnarray*}
\norm{1_{\vert x\vert >\eta_j\lambda_j + \vert t\vert}V_{\{\lambda_k\}}u}_{L^1_tL^2_x} &\lesssim& \norm{1_{\vert x\vert > \eta_j\lambda_j+\vert t\vert} V_{\{\lambda_k\}}}_{L^{\frac{2(N+1)}{N+4}}_tL^{\frac{2(N+1)}{3}}_x}\norm{1_{\vert x\vert > \eta_j\lambda_j+\vert t\vert}u}_{L^{\frac{2(N+1)}{N-2}}_tL^{\frac{2(N+1)}{N-2}}_x}\\
&\lesssim& \norm{1_{\vert x\vert > \eta_j\lambda_j+\vert t\vert} V_{\{\lambda_k\}}}_{L^{\frac{2(N+1)}{N+4}}_tL^{\frac{2(N+1)}{3}}_x}\norm{u_0}_{\dot{H}^1_{\eta_j\lambda_j}}
\end{eqnarray*}
where in the second to last step we used Lemma \ref{multistrich}. 
Now let
\[
y=\frac{x}{\lambda_k},\; s=\frac{t}{\lambda_k}.
\]
Then
\begin{eqnarray*}
\norm{1_{\vert x\vert >\eta_j\lambda_j + \vert t\vert}V_{\{\lambda_k\}}}_{L^{\frac{2(N+1)}{N+4}}_tL^{\frac{2(N+1)}{3}}_x} &=& \frac{1}{\lambda_k^2} \lambda_k^{\frac{N+4}{2(N+1)}+\frac{3N}{2(N+1)}}\norm{\chi_{\{\vert y\vert \geq \frac{\eta_j\lambda_j}{\lambda_k} + \vert s\vert\}}(y) (1+\vert y\vert)^{-4}}_{L_s^{\frac{2(N+1)}{N+4}}L_y^{\frac{2(N+1)}{3}}}\\
&=& \norm{\chi_{\{\vert y\vert \geq \frac{\eta_j\lambda_j}{\lambda_k} + \vert s\vert\}}(y) (1+\vert y\vert)^{-4}}_{L_s^{\frac{2(N+1)}{N+4}}L_y^{\frac{2(N+1)}{3}}}
\end{eqnarray*}
Let $R=\eta_j\frac{\lambda_j}{\lambda_k}.$ We will split the integral into $s \in (0,1)$ and $s\in [1,\infty).$ If $s\geq 1,$ then $\vert y\vert>1$ and so $(1+\vert y\vert)=O(\vert  y\vert)$ which implies that
\[
\left(\int_{\vert y\vert > R+s}\frac{dy}{(1+\vert y\vert)^{4p}}\right)^{\frac{q}{p}} \lesssim \left(\int_{\vert y\vert > R+\vert s\vert} \frac{1}{\vert y\vert^{4p}}dy\right)^{\frac{q}{p}} = \left(\int_{R+s}^\infty \frac{r^{N-1}}{r^{4p}}\right)^{\frac{q}{p}} \lesssim \left(\frac{1}{(R+s)^{4p-N}}\right)^{\frac{q}{p}} = \frac{1}{(R+s)^{\frac{5N+8}{N+4}}},
\]
where
\[
q=\frac{2(N+1)}{N+4},\; p = \frac{2(N+1)}{3}.
\]
Now integrating in $s$ we obtain
\[
\norm{1_{\vert s\vert\geq 1}1_{\vert x\vert >\eta_j\lambda_j + \vert t\vert}V_{\{\lambda_k\}}}_{L^{\frac{2(N+1)}{N+4}}_tL^{\frac{2(N+1)}{3}}_x} \lesssim \left(\int_{\R} \frac{1}{(R+\vert s\vert)^{\frac{5N+8}{N+4}}}ds\right)^{\frac{1}{p}} \lesssim \frac{1}{R^2}.
\]
In the case $0 < s<1,$ we split this into the domains $r \geq 1$ and $R<1.$ In the first case, we have that $\vert y\vert > R+\vert s\vert$ and so we can follow the exact computations as in the case $s \geq 1.$ Now if instead $R \leq 1,$ we split into the cases that $\vert y\vert \geq 2$ and $R+\vert s\vert \leq \vert y\vert \leq 2.$ If $\vert y\vert \geq 2,$ then again $1+\vert y\vert =O(\vert y\vert),$ and so we can proceed as in the case that $\vert s\vert \geq 1.$ In the remaining case, we note that $1+\vert y\vert=O(1)$ and so combining these we obtain that
\[
\norm{1_{\vert s\vert\leq 1}1_{\vert x\vert >\eta_j\lambda_j + \vert t\vert}V_{\{\lambda_k\}}}_{L^{\frac{2(N+1)}{N+4}}_tL^{\frac{2(N+1)}{3}}_x} \lesssim  \frac{1}{R^2} +1 \lesssim \frac{1}{R^2}.
\]
This shows that
\[
\norm{1_{\vert s\vert\geq 1}1_{\vert x\vert >\eta_j\lambda_j + \vert t\vert}V_{\{\lambda_k\}}}_{L^{\frac{2(N+1)}{N+4}}_tL^{\frac{2(N+1)}{3}}_x} \lesssim \frac{\lambda_k^2}{\lambda_j^2\eta_j^2}.
\]
After following the same steps, we obtain
\[
\norm{\Pi_{\lambda_j,\eta_j\lambda_j}^\perp u_0} \leq \frac{\e_j}{2} \norm{u_0}_{\dot{H}^1}+C_{\e_j,\eta_j,\e_{j-1},\eta_{j-1}}\sqrt{E_{\mathrm{out}}}
\]
for $N \equiv 0 \mod 4$ and similarly for dimensions $N \equiv 2 \mod 4.$ We can also obtain the bound
\[
\left\vert \int_{\vert x\vert \geq \eta_j\lambda_j} \nabla u_0 \cdot \nabla (\Lambda W)_{(\lambda_j)}\right\vert = o_{\eta_j}(1)\norm{u_0}_{\dot{H}^1}
\]
for $N \equiv 0 \mod4$ and the analogous bound for $N\equiv 2\mod 4.$ Combining everything we obtain
\[
\norm{\bar{P}_{\lambda_j,\eta_j\lambda_j}}_{\dot{H}_{\eta_j\lambda_j}} \leq C_{\eta_j} \left(\e_{j-1} + o_{\eta_{j-1}}(1)+\frac{\gamma(\lambda)}{\eta_j^2}\right)\norm{u_0}_{\dot{H}_1} + C_{\e_j,\eta_j,\e_{j-1},\eta_{j-1}}\sqrt{E_{\mathrm{out}}}.
\]
Hence if $\gamma(\lambda)$ is small enough we obtain
\[
\norm{\bar{P}_{\lambda_j,\eta_j\lambda_j}}_{\dot{H}_{\eta_j\lambda_j}} \leq \frac{\e_j}{2} \norm{u_0}_{\dot{H}^1}+C_{\e_j,\eta_j,\e_{j-1},\eta_{j-1}} \sqrt{E}_{\mathrm{out}},
\]
with an analogous statement for $N \equiv 2 \mod 4.$ All that remains is to bound the projection onto $\{T_k^\infty\}_{k=1}^{\frac{N-6}{2}}.$ We can see by a contradiction argument (as in case 4 of the proof of Lemma \ref{notinterestingsingle}) that there exists an absolute constant $C$ such that
\[
\norm{\bar{P}_{\lambda_j,\eta_j\lambda_j}u_0}_{\dot{H}_{\eta_j\lambda_j}^1} \geq C\norm{u_0}_{\dot{H}^1_{\eta_j\lambda_j}}.
\]
Now we have
\begin{eqnarray*}
\norm{u_0}_{\dot{H}^1_{\eta_j\lambda_j}}  \leq C^{-1}\norm{\bar{P}_{\lambda_j,\eta_j\lambda_j}u_0}_{\dot{H}_{\eta_j\lambda_j}^1} \leq C^{-1} \left(\frac{\e_j}{2} \norm{u_0}_{\dot{H}^1}+C_{\e_j,\eta_j,\e_{j-1},\eta_{j-1}} \sqrt{E_{\mathrm{out}}} \right).
\end{eqnarray*}
This completes the induction. Now to finish the proof, we set $\eta_J = 0$ in the final step.
\end{proof}
\begin{comment}
\textcolor{red}{however something has to be done to bound the projections to $T_j^\infty$ for $1\leq j \leq \frac{N-6}{2}.$ Perhaps it is possible with an argument as follows: one can try to close the induction argument above with the norm $\norm{\Pi^{\perp,*}u_1}_{\dot{H}^1}$ where $\Pi^{\perp,*}$ is the projection to the orthogonal of $\{T_k^\infty\}_{k=1}^{\frac{N-6}{2}}.$ Then to bound the projections to $\{T_k^\infty\}_{k=1}^{\frac{N-6}{2}}$ one can try to establish a bound of the form
\[
\norm{\Pi^* u }_{\dot{H}^1_{\rho}} \leq C\norm{u_0}_{\dot{H}^1_{\rho}}
\]
for any small enough radius $\rho,$ with a constant $C$ independent of $\rho.$ To do so, one can use a contradiction argument, and proceed similar to case 4 in the proof of Lemma \ref{notinterestingsingle}.}
\end{comment}

\subsection{Resonant Component}
In this section we prove the other half of the energy estimate in Theorem \ref{mainmulti}.
\begin{prop}\label{p67}
For all $J \in \N,$ there exist $\gamma^*,C>0$ such that for all $\lambda \in \Lambda_J$ with $\gamma(\lambda)\leq \gamma^*$ and any radial solution of \eqref{multieq} satisfies
\[
\norm{\Pi_{L^2,\lambda}^\perp u_1}_{Z_{-3,\lambda}}^2 \leq C \left(\sum_{\pm} \lim_{t\to \pm \infty} \int_{r\geq \vert t\vert} \vert \nabla_{t,x} u \vert^2 dx + \gamma(\lambda)\norm{u_1}_{Z_{-3,\lambda}}^2\right)
\]
if $N \equiv 0 \mod 4,$ and similarly
\[
\norm{\nabla \Pi_{\dot{H}^1,\lambda}^\perp u_0}_{Z_{-3,\lambda}}^2 \leq C \left(\sum_{\pm} \lim_{t\to \pm \infty} \int_{r\geq \vert t\vert} \vert \nabla_{t,x} u \vert^2 dx + \gamma(\lambda)\norm{\nabla u_0}_{Z_{-3,\lambda}}^2\right)
\]
if $N \equiv 2 \mod 4.$
\end{prop}
\begin{proof}
We will prove it for the case $N \equiv 2 \mod 4,$ with the $N \equiv 0 \mod 4$ case being identical. We claim it suffices to prove that
\begin{equation}\label{int2}
\norm{\Pi_{L^2,\lambda}^\perp u_1}_{Z_{-3,\lambda}} \lesssim \sqrt{E_{\mathrm{out}}} \mbox{ if } u_1 = \Pi_{L^2,\lambda}^\perp u_1
\end{equation}
Indeed, suppose these hold. Let $u$ be a solution to \eqref{multieq}. We can assume by time symmetry that
\[
(u_0,u_1)=(0,u_1)
\]
if $N \equiv 0 \mod 4,$ and
\[
(u_0,u_1)=(u_0,0)
\]
if $N \equiv 2 \mod 4.$ We write $u_0=v_0+w_0$ for $w_0 = \Pi_{\dot{H}^1,\lambda}^\perp u_0,$ and similarly $u_1=v_1+w_1$ with $w_1 = \Pi_{L^2,\lambda}^\perp u_1.$ Then we claim that
\[
\lim_{t\to\infty} \int_{\vert x\vert > \vert t\vert} \vert \nabla_{t,x}v(t,x)\vert^2 dx \lesssim
\left\{
\begin{matrix}
\gamma(\lambda)\norm{v_1}_{Z_{-3,\lambda}}^2 & \mbox{ if } N \equiv 0 \mod 4,\\
\gamma(\lambda)\norm{\nabla v_0}_{Z_{-3,\lambda}}^2 & \mbox{ if } N \equiv 2 \mod 4.
\end{matrix}
\right.
\]
If this claim holds then
\[
\norm{w_1}_{Z_{-3,\lambda}}^2 \lesssim E_{\mathrm{out}} + \lim_{t\to \infty} \int_{\vert x\vert > \vert t\vert} \vert \nabla_{t,x} v(t,x)\vert^2 dx \lesssim E_{\mathrm{out}} + \gamma(\lambda) \norm{v_1}_{Z_{-3,\lambda}}^2 \lesssim E_{\mathrm{out}} + \gamma(\lambda)\left(\norm{u_1}_{Z_{-3,\lambda}}^2 + \norm{u_1}_{Z_{-3,\lambda}}^2\right)
\]
for $N \equiv 0 \mod 4,$ and similarly for $N \equiv 2\mod 4.$ To show \eqref{int2}, note that if $j_0$ is such that $\vert \alpha_{j_0}\vert = \sup_{1\leq j \leq J} \vert \alpha_j\vert$ where
\[
v_1 = \sum_{j=1}^J \alpha_j (\Lambda W)_{[\lambda_j]}, \; v_0 = \sum_{j=1}^J \alpha_j (\Lambda W)_{(\lambda_j)}.
\]
We obtain
\begin{eqnarray*}
\norm{v_1}_{Z_{-3,\lambda}} &\gtrsim& \frac{1}{\inf_{1\leq j \leq J}\left\langle \log \frac{\lambda_{j_0}}{\lambda_j}\right\rangle} \left(\int_{\lambda_{j_0}\leq \vert x\vert \leq 2\lambda_{j_0}} \vert v_1\vert^2 dx\right)^{\frac{1}{2}}\\
&\gtrsim &\vert\alpha_{j_0}\vert \frac{1}{\inf_{1\leq j \leq J}\left\langle \log \frac{\lambda_{j_0}}{\lambda_j}\right\rangle} \left(\int_{\lambda_{j_0}\leq \vert x\vert \leq 2\lambda_{j_0}} \vert \Lambda W_{[\lambda_{j_0}]}\vert^2 dx\right)^{\frac{1}{2}}\\
&\gtrsim & \vert \alpha_{j_0}\vert
\end{eqnarray*}
and similarly $\norm{\nabla v_0}_{Z_{-3,\lambda}} \geq \vert \alpha_{j_0}\vert.$ On the other hand, we have
\[
\lim_{t\to\infty} \int_{\vert x\vert \geq \vert t\vert} \vert \nabla_{t,x} v(t,x)\vert^2 dx \lesssim \lim_{t\to\infty} \int_{\vert x\vert \geq \vert t\vert} \left\vert \nabla_{t,x} \left(\sum_{j=1}^J t\alpha_j (\Lambda W)_{[\lambda_j]}\right)\right\vert^2 dx + \gamma(\lambda) \sup_{1\leq j \leq J} \alpha_j^2
\]
and similarly for $N \equiv 2 \mod 4.$ This proves the claim.

Now let $\gamma := \gamma(\lambda),$ and define the radii
\[
R_j^+ = \sqrt{\lambda_j\lambda_{j+1}}, \; R_j^- = \frac{\lambda_{j+1}}{\gamma^{\frac{1}{4}}}.
\]
Note that $R_j^- \leq \gamma^{\frac{1}{4}} R_j^+.$ For any $j,$ let
\[
\norm{f}_{Z_{-3,\lambda_j}} := \norm{f_{\left[\frac{1}{\lambda_j}\right]}}_{Z_{-3}}.
\]
Then
\[
\norm{f}_{Z_{-3,\lambda_j}} \lesssim \norm{f}_{-3,\lambda},
\]
and if $\mathrm{Supp} f \subset [R_j^+ , \infty),$ then
\[
\norm{f}_{Z_{-3,\lambda}} \lesssim \norm{f}_{Z_{-3,\lambda_j}}.
\]
We will prove the following statement by induction.
\begin{claim}\label{multiintind}
There is an absolute constant $\alpha>0$ such that for all $j \in \{1, \ldots , J\},$
\[
\norm{1_{r \geq R_j^+}u_1}_{Z_{-3,\lambda_j}} \lesssim \sqrt{E_{\mathrm{out}}} + \gamma^\alpha \norm{u_1}_{Z_{-3,\lambda}}
\]
if $N \equiv 0 \mod 4,$ and similarly
\[
\norm{1_{r \geq R_j^+}\nabla u_0}_{Z_{-3,\lambda_j}} \lesssim \sqrt{E_{\mathrm{out}}} + \gamma^\alpha \norm{\nabla u_0}_{Z_{-3,\lambda}}
\]
if $N \equiv 2 \mod 4.$
\end{claim}
We will prove the statement for $N \equiv 2 \mod 4,$ the other case being identical.

Suppose that $2 \leq J \leq J-1,$ and supppose Claim \ref{multiintind} holds for $j-1.$ Let $v_j,\tilde{v}_j$ be the solutions to
\[
\left\{
\begin{matrix}
\p_t^2 v_j - \Delta v_j + V_{\{\lambda_j\}}v_j = 0,\\
(v_j(0),\p_t v_j(0))= (u,\p_t u)\vert_{t=0},
\end{matrix}
\right.
\]
\[
\left\{
\begin{matrix}
\p_t^2 \tilde{v}_j - \Delta \tilde{v}_j + V_\lambda \tilde{v}_j = F(v_j) = -\sum_{i \neq j} V_{\{\lambda_i\}} v_j,\\
(\tilde{v}_j(0),\p_t \tilde{v}_j(0))=(0,0).
\end{matrix}
\right.
\]
Let $\chi$ be a smooth cut off with
\[
\chi(r) =
\left\{
\begin{matrix}
1 \mbox{ if } r \leq 1,\\
0 \mbox{ if } r \geq 2.
\end{matrix}
\right.
\]
Let $v_j^m$ be solutions to
\[
\p_t^2 v_j^m - \Delta v_j^m + V_{\{\lambda_i\}} v_j^m = 0
\]
with
\[
v_j^1(0) = \chi_{R_{j-1}^+} (1-\chi_{R_j^-})u_0 , v_j^2(0) = (1-\chi_{R_{j-1}^+})u_0, v_j^3(0) = \chi_{\frac{R_j^-}{48}} u_0
\]
where $\chi_\rho = \chi\left(\frac{\cdot}{\rho}\right).$ This implies that
\[
v_j=v_j^1+v_j^2 + v_j^3,
\]
and so
\[
F(v_j) = F(v_j^1) + F(v_j^2) + F(v_j^3).
\]
\begin{lemma}\label{step4}
For some absolute constant $\alpha>0,$ w   e have
\begin{equation}\label{438}
\norm{1_{\left\{r\geq \frac{R_j^-}{24}+\vert t\vert\right\}}F(v_j^1)}_{L^1_tL^2_x} \lesssim
\gamma^\alpha \norm{\nabla u_0}_{Z_{-3,\lambda}}
\end{equation}
if $N\equiv 2 \mod 4$ and similarly
\[
\norm{1_{\left\{r\geq \frac{R_j^-}{24}+\vert t\vert\right\}}F(v_j^1)}_{L^1_tL^2_x} \lesssim
\gamma^\alpha \norm{u_1}_{Z_{-3,\lambda}}
\]
in the case that $N \equiv 0 \mod 4,$
\begin{equation}\label{439}
\norm{1_{\left\{r\geq \frac{R_j^-}{24}+\vert t\vert\right\}} F(v_j^2)}_{L_t^1L_x^2}\lesssim \sqrt{E_{\mathrm{out}}} + \gamma^\alpha \norm{\nabla u_0}_{Z_{-3,\lambda}}
\end{equation}
if $N\equiv 2\mod 4$ with the analogous statement for $N \equiv 0\mod 4,$ and
\begin{equation}\label{trivialeq}
\norm{1_{\{r \geq \frac{R_j^-}{24}+\vert t\vert\}}F(v_j^3)}_{L_t^1L_x^2} = 0.
\end{equation}
\end{lemma}
We will prove this lemma later, and will continue the proof assuming its validity. By Lemma \ref{multistrich}, we have
\[
\sup_t \norm{(\tilde{v}_j(t),\p_t \tilde{v}_j(t))}_{\mathcal{H}_{R_j^-+\vert t\vert}} \lesssim \sqrt{E_{\mathrm{out}}} + \gamma^\alpha\norm{\nabla u_0}_{Z_{-3,\lambda}}.
\]
Indeed,
\[
\sup_{t}\norm{(\tilde{v}_j(t),\p_t\tilde{v}_j(t))}_{\mathcal{H}_{R_j^-+\vert t\vert}} \lesssim \norm{F(v_j)}_{L_t^1L_x^2(\{r \geq R_j^- + \vert t\vert\})} \lesssim \sqrt{E_{\mathrm{out}}} + \gamma^\alpha \norm{\nabla u_0}_{Z_{-3,\lambda}}
\]
where we used Lemma \ref{step4}. Therefore
\[
\sum_{\pm}\lim_{t\to\infty} \int_{\vert x\vert \geq \frac{R_j^-}{24}+\vert t\vert} \vert \nabla_{t,x}v_j(t)\vert^2 dx \lesssim E_{\mathrm{out}} + \gamma^{2\alpha} \norm{\nabla u_0}_{Z_{-3,\lambda}}^2.
\]
Now by the $N\equiv 2 \mod 4$ version of Lemma \ref{56equiv} there exist coefficients $c_{j,k}^\infty$ such that
\begin{equation}\label{eq106}
u_0 = \sum_{k=0}^{\frac{N-6}{2}} c_{j,k}^\infty (T_k^\infty)_{(\lambda_j)} + \bar{u}_{0,j}
\end{equation}
with
\begin{equation}\label{eq107}
\norm{1_{r\geq R_j^-}\nabla \bar{u}_{0,j}}_{Z_{-3,j}}\lesssim \sqrt{E_{\mathrm{out}}} + \gamma^\alpha\norm{\nabla u_0}_{Z_{-3,\lambda}}.
\end{equation}
Taking the $\dot{H}^1(\{R_j^- \leq \vert x\vert \leq 2R_j^-\})$ norm, we see that that
\begin{eqnarray*}
\int_{R_j^-}^{2R_j^-} \left\vert \nabla \left(\sum_{k=1}^{\frac{N-6}{2}} c_{j,k}^\infty (T_k^\infty)_{(\lambda_j)}\right)\right\vert^2 r^{N-1}dr &\lesssim& \int_{R_j^-}^{2R_j^-} \vert \nabla u_0\vert^2 r^{N-1} dr + \int_{R_j^-}^{2R_j^-} \vert \nabla \bar{u}_{0,j}\vert^2 r^{N-1}dr\\
&+& (c_{j,0}^\infty)^2 \int_{R_j^-}^{2R_j^-} \vert \nabla((\Lambda W)_{(\lambda_j)})\vert^2 r^{N-1} dr.
\end{eqnarray*}
Note that
\[
\int_{R_j^-}^{2R_j^-} \left\vert\nabla \left(\sum_{k=1}^{\frac{N-6}{2}} c_{j,k}^\infty (T_k^\infty)_{(\lambda_j)}\right)\right\vert^2 r^{N-1}dr = \sum_{k_1,k_2 \in \{1,\ldots , \frac{N-6}{2}\}} c_{j,k_1}^\infty c_{j,k_2}^\infty \int_{R_j^-}^{2R_j^-} \nabla (T_{k_1}^\infty)_{(\lambda_j)} \cdot \nabla (T_{k_2}^\infty)_{(\lambda_j)} r^{N-1}dr.
\]
Let $M = (m_{k_1k_2})_{1\leq k_1,k_2 \leq \frac{N-6}{2}}$ where
\[
m_{k_1k_2} :=\int_{R_j^-}^{2R_j^-} \nabla (T_{k_1}^\infty)_{(\lambda_j)} \cdot \nabla (T_{k_2}^\infty)_{(\lambda_j)} r^{N-1}dr
\]
Using that
\[
\vert (T_k^\infty)_{[\lambda_j]}\vert \sim c_k \frac{1}{(R_j^-)^{\frac{N}{2}}}\left(\frac{\lambda_j}{\lambda_{j+1}}\right)^{\frac{N-4}{2}} \gamma^{\frac{N-4}{8}}
\]
if $r \sim R_j^-$ we see that
\[
m_{k_1k_2}\sim c_{k_1}c_{k_2} \left(\frac{\lambda_j}{\lambda_{j+1}}\right)^{N-4} \gamma^{\frac{N-4}{4}} (R_j^-)^2
\]
and so for some $O(1)$ matrix $\tilde{M},$
\begin{eqnarray*}
&&\langle \tilde{M}(c_{j,k}^\infty)_{1\leq k \leq \frac{N-6}{2}},(c_{j,k}^\infty)_{1\leq k \leq \frac{N-6}{2}}\rangle \left(\frac{\lambda_j}{\lambda_{j+1}}\right)^{N-2} \gamma^{\frac{N-6}{4}}(R_j^-)^{N-6}\\
&&\lesssim \langle \log \gamma\rangle^2 \norm{\nabla u_0}_{Z_{-3,\lambda}}^2 + \left\langle \log \frac{\lambda_{j+1}}{\lambda_j\gamma^{\frac{1}{4}}}\right\rangle^2 \norm{1_{r\geq R_j^-}\nabla \bar{u}_{0,j}}_{Z_{-3,\lambda_j}}^2 + (c_{j,0}^\infty)^2 \left(\frac{\lambda_{j+1}^{\frac{N+2}{2}}}{\gamma^{\frac{N+2}{8}} \lambda_j^{\frac{N+2}{2}}}\right)^2.
\end{eqnarray*}
where we used that
\[
\vert \nabla (\Lambda W)_{(\lambda_i)}\vert \sim \frac{1}{\lambda_i^{\frac{N}{2}}} \left\vert(\nabla \Lambda W) \left(\frac{r}{\lambda_i}\right)\right\vert \sim \left(\frac{\lambda_{j+1}}{\lambda_j}\right)^{\frac{N+2}{2}}\gamma^{-\frac{N+2}{8}} R_j^{-\frac{N}{2}}.
\]
This implies that
\begin{eqnarray}\label{cfromd}
&&\left\langle \tilde{M}(c_{j,k}^\infty)_{1\leq k \leq \frac{N-6}{2}},(c_{j,k}^\infty)_{1\leq k \leq \frac{N-6}{2}}\right\rangle\nonumber\\
&&\lesssim \left(\frac{\lambda_{j+1}}{\lambda_j}\right)^{N-6} \gamma^{-\frac{N-6}{4}} (R_j^-)^{-(N-6)} \left(\langle \log r\rangle^2 \norm{\nabla u_0}_{Z_{-3,\lambda}}^2 + \left\langle \log \frac{\lambda_{j+1}}{\lambda_j \gamma^{\frac{1}{4}}}\right\rangle^2 \norm{1_{r\geq R_j^-}\nabla \bar{u}_{0,j}}_{Z_{-3,\lambda}}^2\right)\nonumber\\
&&+ (c_{j,0}^\infty)^2 \left(\frac{\lambda_{j+1}^{\frac{N+2}{2}}}{\gamma^{\frac{N+2}{8}} \lambda_j^{\frac{N+2}{2}}}\right)^2\\
&& \lesssim \frac{1}{\sqrt{\gamma}} \frac{\lambda_{j+1}^2}{\lambda_j^2}\log\left(\frac{1}{\sqrt{\gamma}} \frac{\lambda_{j+1}^2}{\lambda_j^2}\right) \left(\frac{1}{\gamma^{\frac{N+2}{8}}} \left(\frac{\lambda_{j+1}}{\lambda_j}\right)^{\frac{N+2}{2}}\vert c_{j,0}^\infty\vert + \norm{\nabla u_0}_{Z_{-3,\lambda}} + \sqrt{E_{\mathrm{out}}}\right).
\end{eqnarray}
On the other hand, we have $(u_0,u_1)\perp_{\dot{H}^1\times L^2} ((\Lambda W)_{(\lambda_i)},(\Lambda W)_{[\lambda_i]}).$ Using that
\[
\vert \nabla \Lambda W(r)\vert \lesssim r, \; R_j^- \lesssim \gamma(\lambda)^{\frac{3}{4}} \lambda_j,
\]
we have
\begin{equation}\label{semiorth}
\int_{R_j^- \leq \vert x\vert} \nabla u_0 \cdot \nabla(\Lambda W)_{(\lambda_j)} = \int_{\vert x\vert \leq R_j^-}\nabla u_0 \cdot \nabla(\Lambda W)_{(\lambda_j)} \lesssim O(\gamma^3 \vert \log \gamma\vert \norm{\nabla u_0}_{Z_{-3,\lambda}}).
\end{equation}
Now using that
\[
\int_{\vert x\vert \leq \frac{R_j^-}{\lambda_j}} \vert \nabla \Lambda W\vert^2 dx \lesssim\int_{\vert x\vert \leq \frac{R_j^-}{\lambda_j}} \vert x\vert^2 dx \lesssim \left(\frac{R_j^-}{\lambda_j}\right)^{N+2} \lesssim \frac{1}{\gamma^{\frac{N+2}{4}}}\gamma^{N+2} = \gamma^{\frac{3}{4}(N+2)}.
\]
We also see that
\begin{equation}\label{otherorth}
\int_{R_j^-\leq \vert x\vert} \nabla u_0 \cdot \nabla (\Lambda W)_{(\lambda_j)} = c_{j,0}^\infty \int_{\R^n} \vert \nabla \Lambda W\vert^2 \left(1+O\left(\gamma^{\frac{3}{4}(N+2)}\right)\right) + O\left(\sum_{k=1}^{\frac{N-6}{2}} \vert c_{j,k}\vert + \sqrt{E_{\mathrm{out}}} + \gamma^\alpha \norm{\nabla u_0}_{Z_{-3,\lambda}}\right).
\end{equation}
This implies that
\begin{equation}\label{dfromc}
\vert c_{j,0}^\infty\vert \int \vert \nabla \Lambda W\vert^2 \left(1+O\left(\gamma^{\frac{3}{4}(N+2)}\right)\right) \leq \sum_{k=1}^{\frac{N-6}{2}} \vert c_{j,k}\vert + \sqrt{E_{\mathrm{out}}} + \gamma^\alpha \norm{\nabla u_0}_{Z_{-3,\lambda}}
\end{equation}
where we used \eqref{semiorth}, \eqref{otherorth}. Combining this with \eqref{cfromd} we obtain that
\[
\vert c_{j,0}^\infty\vert \lesssim \sqrt{E_{\mathrm{out}}} + \gamma^\alpha \norm{\nabla u_0}_{Z_{-3,\lambda}}
\]
and using that
\[
\left\langle \tilde{M} (c_{j,k}^\infty)_{1\leq j \leq \frac{N-6}{2}},(c_{j,k}^\infty)_{1\leq j \leq \frac{N-6}{2}}\right\rangle \gtrsim \sum_{k=1}^\infty \vert c_{j,k}\vert
\]
we obtain that (since $\frac{N+2}{2} > \frac{N+2}{8}$ and $\gamma$ is small)
\begin{eqnarray*}
\sum_{k=1}^\infty \vert c_{j,k}\vert &\lesssim& \frac{1}{\sqrt{\gamma}} \left(\frac{\lambda_{j+1}}{\lambda_j}\right)^2 \log \left(\frac{1}{\sqrt{\gamma}}\frac{\lambda_{j+1}^2}{\lambda_j^2}\right)(\norm{\nabla u_0}_{Z_{-3,\lambda}} + \sqrt{E_{\mathrm{out}}})\\
&\lesssim & \frac{1}{\gamma^{\frac{3}{8}}} \frac{\lambda_{j+1}^{\frac{3}{2}}}{\lambda_j^{\frac{3}{2}}} (\norm{\nabla u_0}_{Z_{-3,\lambda}} + \sqrt{E_{\mathrm{out}}}).
\end{eqnarray*}
Since $\norm{\nabla (\Lambda W)_{(\lambda_i)}}_{Z_{-3,\lambda_i}} \lesssim 1,$ we have
\begin{equation}\label{infty}
\norm{c_{j,0}^\infty \nabla(\Lambda W)_{(\lambda_i)}}_{Z_{-3,\lambda_i}} \lesssim \vert c_{j,0}^\infty\vert \lesssim \sqrt{E_{\mathrm{out}}} + \gamma^\alpha \norm{\nabla u_0}_{Z_{-3,\lambda}}.
\end{equation}
We also have for $r \leq R_j^+$
\[
\vert \nabla (T_k^\infty)_{(\lambda_j)}\vert \lesssim \frac{1}{\lambda_j^{\frac{N}{2}}} \left(\frac{r}{\lambda_j}\right)^{-(N-1)} = \frac{\lambda_j^{\frac{N}{2}-1}}{r^{N-1}} \lesssim \left(\frac{\lambda_j}{R_j^+}\right)^{\frac{N}{2}-1} \frac{1}{r^{\frac{N}{2}}} = \frac{1}{r^{\frac{N}{2}}} \left(\frac{\lambda_j}{\lambda_{j+1}}\right)^{\frac{N-2}{4}}
\]
and so
\begin{eqnarray}\label{inftyother}
\norm{1_{r\geq R_j^+} c_{j,k}^\infty \nabla (T_k^\infty)_{(\lambda_j)}}_{Z_{-3,\lambda_j}} &\lesssim& \vert c_{j,k}^\infty\vert \norm{1_{\{r\geq R_j^+\}} \frac{1}{r^{\frac{N}{2}}} \left(\frac{\lambda_j}{\lambda_{j+1}}\right)^{\frac{N-2}{4}}}_{Z_{-3,\lambda_j}}\nonumber\\
&\lesssim & \vert c_{j,k}^\infty\vert \left(\frac{\lambda_j}{\lambda_{j+1}}\right)^{\frac{N-2}{4}} \norm{1_{\{r \geq R_j^+\}}\frac{1}{r^{\frac{N}{2}}}}_{Z_{-3,\lambda_j}}\nonumber\\
&\lesssim & \frac{1}{\gamma^{\frac{3}{8}}} \left(\frac{\lambda_j}{\lambda_{j+1}}\right)^{-\frac{1}{2}} (\norm{\nabla u_0}_{Z_{-3,\lambda}}+\sqrt{E_{\mathrm{out}}})\nonumber\\
&\lesssim & \gamma^{\frac{1}{8}}(\norm{\nabla u_0}_{Z_{-3,\lambda}}+\sqrt{E_{\mathrm{out}}})
\end{eqnarray}
Combining \eqref{infty}, \eqref{inftyother}, \eqref{eq106}, \eqref{eq107} we obtain
\[
\norm{1_{\{r \geq R_j^+\}}\nabla u_0}_{Z_{-3,\lambda_j}} \lesssim \sqrt{E_{\mathrm{out}}} + \gamma^{\alpha} \norm{\nabla u_0}_{Z_{-3,\lambda}}
\]
By the induction hypothesis we also have
\[
\norm{1_{\{r \geq R_{j-1}^+\}}\nabla u_0}_{Z_{-3,\lambda_{j-1}}} \lesssim \sqrt{E_{\mathrm{out}}} + \gamma^{\alpha} \norm{\nabla u_0}_{Z_{-3,\lambda}}
\]
and so using
\[
\norm{1_{\{r>R_j^+\}}f}_{Z_{-3,\lambda_j}} \lesssim \norm{1_{\{r\geq R_{j-1}^+\}}f}_{Z_{-3,\lambda_{j-1}}} + \norm{1_{\{r>R_j^+\}}f}_{Z_{-3,\lambda_j}}
\]
we complete the induction.

Now for $j=1,\ldots , J,$ there are slight changes. If $j=1,$ we set up $R_{j-1}^-=R_{j-1}^+=\infty$ and $v_j^2=0.$

For $j=J,$ we set $c_{j,k}^\infty = 0$ if $k \geq 1$ and $R_j^-=R_J^+=0.$ This completes the proof of Proposition \ref{p67} assuming Lemma \ref{step4}, which we will now prove.
\begin{proof}[Proof of Lemma \ref{step4}]
First note that
\[
\norm{1_{\{r\geq \frac{R_j^-}{24}+\vert t\vert\}} F(v_j^3)}_{L_t^1L_x^2} = 0
\]
by finite propagation speed. Also, suppose $1\leq j \leq J$ with $i\neq j.$ Then
\[
\norm{\nabla v_1^j(0)}_{Z_{-3,\lambda_j}} \lesssim \norm{\nabla u_0}_{Z_{-3,\lambda}}.
\]
If $i>j,$ then note that by definition of $v_j^1,$
\[
v_j^1(0,r)=0 \mbox{ if }r \leq \frac{\lambda_i}{48\gamma^{\frac{1}{4}}} \leq \frac{R_j^-}{48}.
\]
Now by Lemma \ref{L43}, we have
\[
\norm{1_{r\geq \vert t\vert} V_{(\lambda_i)}v_j^1}_{L_t^1L_x^2} \lesssim \gamma^{\frac{1}{4}}\norm{\nabla v_j^1}_{Z_{-3,\lambda_i}} \lesssim \gamma^{\frac{1}{4}}\norm{\nabla u_0}_{Z_{-3,\lambda}}.
\]
If instead $i<j,$ then by definition of $v_j^1,$ $v_j^1(0,r)=0$ if $r \geq 2\sqrt{\gamma}\lambda_i \geq 2R_{j-1}^+.$ Then by Lemma \ref{L43} we obtain
\[
\norm{1_{\{r\geq \vert t\vert\}}V_{\{\lambda_i\}}v_j^1}_{L_t^1L_x^2} \leq \gamma^{\frac{3}{8N}} \norm{\nabla u_0}_{Z_{-3,\lambda}}.
\]
These imply \eqref{438}. To show \eqref{439}, suppose $1\leq i \leq J,$ $i\neq j.$ Since
\[
\norm{\nabla((1-\chi_R^0)f)}_{Z_{-3,\lambda}}\leq C\norm{1_{\vert x\vert \geq R}\nabla f}_{Z_{-3,\lambda}},
\]
we have
\[
\norm{\nabla v_j^2(0)}_{Z_{-3,\lambda_{j-1}}} \leq C\norm{1_{\{r \geq R_{j-1}^+\}}\nabla u_0}_{Z_{-3,\lambda_{j-1}}} \lesssim \sqrt{E_{\mathrm{out}}} + \gamma^\alpha \norm{\nabla u_0}_{Z_{-3,\lambda}}.
\]
Since $v_j^2(0,r)=0$ if $r \leq R_{j-1}^+$ we have
\[
\norm{\nabla v_j^2(0)}_{Z_{-3,\lambda_j}} \lesssim \norm{\nabla v_j^2(0)}_{Z_{-3,\lambda_{j-1}}}.
\]
Hence
\[
\norm{\nabla v_j^2(0)}_{Z_{-3,\lambda_i}} \lesssim \sqrt{E_{\mathrm{out}}} + \gamma^\alpha \norm{\nabla u_0}_{Z_{-3,\lambda}}.
\]
Using Lemma \ref{L43}, we obtain
\[
\norm{1_{\{r\geq \vert t\vert\}}V_{\{\lambda_i\}}v_j^2}_{L_t^1L_x^2} \lesssim \sqrt{E_{\mathrm{out}}} + \gamma^\alpha \norm{\nabla u_0}_{Z_{-3,\lambda}}
\]
thus proving \eqref{439}.
\end{proof}
\end{proof}

\bibliography{ref}
\bibliographystyle{plain}
\end{document}